\documentclass[11pt]{amsart}
\usepackage{amssymb, latexsym, amsmath, amsfonts, tikz}
\usepackage{hyperref, enumitem, fancyhdr}
\usepackage{color}

\newtheorem{thm}{Theorem}[section]

\newtheorem{lem}[thm]{Lemma}
\newtheorem{prop}[thm]{Proposition}
\theoremstyle{definition}

\theoremstyle{remark}
\newtheorem{rem}[thm]{Remark}
\numberwithin{equation}{section}
\theoremstyle{remark}

\newcommand{\mbb}{\mathbb}
\newcommand{\ra}{\rightarrow}

\newcommand{\pa}{\partial}
\newcommand{\ov}{\overline}

\newcommand{\ep}{\epsilon}
\newcommand{\no}{\noindent}

\newcommand{\cal}{\mathcal}

\newcommand{\la}{\lambda}

\newcommand{\si}{\sigma}

\newcommand{\cord}{(z_1,z_2, \hdots ,z_k)}

\begin{document}

	\title{Rigidity of Julia sets  of families of biholomorphic mappings in higher dimension}
	\author{Sayani Bera, Ratna Pal}
	\address{SB: School of Mathematics, Ramakrishna Mission Vivekananda Educational and Research Institute, PO Belur Math, Dist. Howrah, 
		West Bengal 711202, India}
	\email{sayanibera2016@gmail.com}
	\address{RP: Department of Mathematics, Indian Institute of Science Education and Research, Pune, Maharashtra-411008, India}
	\email{ratna.math@gmail.com}
	%
	%
	%
	\begin{abstract}
	
	The goal of this article is to study a rigidity property of Julia sets of certain classes of automorphisms in $\mbb C^k$, $k \ge 3.$ First, we study the relation between two polynomial shift--like maps in $\mbb C^k$, $k \ge 3$, that share the same backward and forward Julia sets (or non-escaping sets). Secondly, we study the relation between any pair of skew products of H\'{e}non maps in $\mbb C^3$ having the same forward and backward Julia sets. 
	\end{abstract}
	
	\maketitle
	
	\section{Introduction}
	The fundamental dichotomy of a dynamical system comes from splitting the ambient space into the Fatou set and the Julia set. The Fatou set is an open set where the dynamics is tame and the Julia set is the complement of the Fatou set which supports the wild behavior of the dynamical system. The simplest example of dynamical system which exhibit a rich dynamical behaviour, comes from the class of polynomials in $\mathbb{C}$ of degree greater than or equal to $2$. Being the hub of chaos of the dynamical system, most of the Julia sets of polynomials in $\mathbb{C}$, have very complicated fractal structure. A little perturbation of a polynomial map can change the structure of the corresponding Julia set drastically which gives a hint towards the fact that the Julia sets are quite a rigid object. In fact, a result by Beardon (\cite{Be1}) validates this anticipation which states that for two polynomials $P$ and $Q$ (of degree greater than or equal to $2$), if $J_P=J_Q$  where $J_P$ and $J_Q$ are the Julia sets of $P$ and $Q$ respectively, then 
	\[
	P\circ Q=\sigma \circ P \circ Q
	\]
	where $\sigma(z)=az+b$ with $\lvert a \rvert=1$ and $\sigma(J_P)=J_P$. Note that it was known for long time that if two polynomials $P$ and $Q$ commute, then their Julia sets coincide.  The results obtained in \cite{BE}, \cite{Be2}, \cite{SS}, \cite{LP}  are also pertinent in this direction.
	
	\medskip
	\no 
	In dimension $2$, an analogue of this kind of {\it{rigidity}} property of the Julia sets of H\'{e}non maps has recently been established in \cite{rigidity} (Theorem 1.1) 
	which shows that if we start with two H\'{e}non maps $H$ and $F$ for which both forward and backward Julia sets  coincide, i.e., $J_H^\pm=J_F^\pm$, then 
	\[
	F\circ H=C\circ H\circ F
	\]
	where $C(x,y)=(\delta_- x, \delta_+y)$ with $\lvert\delta_\pm\rvert=1$ and $C(J_H^\pm)=J_H^\pm$.
	
	\medskip 
	\no 
	In this note, we study  {\it{rigidity}} property of Julia sets of some special classes of biholomorphic mappings in $\mathbb{C}^k$, for $k\geq 3$.  The first class considered here  is the class of shift--like polynomial maps introduced by Bedford and Pambuccian in \cite{BP}.  Recall that a shift--like map of type $\nu$, where  $1 \le \nu \le k-1$, in $\mbb C^k$ is a map of the form:
	\[ 
	S(z_1,\hdots,z_k)=(z_2,\hdots,z_k,az_1+p(z_{k-\nu+1}))
	\]
	with $p$ a polynomial in $\mathbb{C}$ of degree greater than or equal to $2$  and with $0\neq a \in \mathbb{C}$. Note that the class of shift--like polynomial maps is a generalization of H\'{e}non maps for dimensions $k>2$.  Further for a polynomial $\nu-$shift $S$, $S^{\nu(k-\nu)}$ or $S^m$ is a regular map where $m$ is a multiple of both $\nu$ and $k-\nu$ (See \cite{mypaper} for a detailed proof).
	
\medskip\no 
First, let us briefly recall a few terminologies associated to shift--like maps. Suppose $S$ is a $\nu-$shift such that $1 \le  \nu \le k-1.$ Recall from \cite{BP} that the filtration corresponding to maps $S$ is as follow:
	\[ V_{R,S}^+=\bigcup_{i=k-\nu+1}^k V^i_{R,S} \text{ and } V_{R,S}^{-}=\bigcup_{i=1}^{k-\nu} V^i_{R,S}\] where 
	\[
	V_{R,S}=\left \{z \in \mathbb C^k: \lvert z_j \rvert\leq R \text{ for } 1\leq j \leq k \right \}
	\]
	and
	\[
	V_{R,S}^i=\left \{ z \in \mathbb C^k \setminus V_{R,S}: \lvert z_i \rvert>\max\{\lvert z_j\rvert, R\} \text{ for } 1 \le j \neq i \le k \right\}.
	\]
	We define non-escaping sets
	\[
	K_S^{\pm} = \{z \in \mathbb C^k : \;\text{the sequence}\; \left(S^{\pm n}(z) \right) \; \text{is bounded} \},
	\]
	and the escaping sets
	\begin{equation}\label{escape}
	U_S^\pm=\mathbb{C}^k \setminus K_S^\pm=\bigcup_{n=0}^\infty S^{\mp n}(V_{R,S}^\pm).
	\end{equation}
	Further, note that $K_S^\pm \subset V_R \cup V_{R,S}^\mp$.
	  The Green functions 
	\[
	G^+_S(z) = \lim_{n \rightarrow \infty} \frac{1}{d^n} \log^+\left \Vert S^{ n\nu }(z) \right \Vert
	\]
	and 
	\[
	G^-_S(z) = \lim_{n \rightarrow \infty} \frac{1}{d^n} \log^+ \left \Vert S^{ -n(k-\nu)}(z)\right \Vert
	\]
	are continuous and  plurisubharmonic in $\mathbb{C}^k$ which vanish precisely on $K_S^\pm$. Further, 
	\[
	G_S^\pm(z)=\log \lvert z \rvert + O(1)
	\]
	in $V_{R,S}^\pm$. The {\it{rigidity}} theorem for shift--like polynomial maps in $\mathbb{C}^k$, for $k\geq 3$, can now be stated as:
	
	\begin{thm}\label{rigidity of shifts}
		 Let $S$ be a polynomial shift--like map of type $\nu$, $1 \le \nu \le k-1$ in $\mbb C^k$, $k \ge 3$. Let $T$ be another shift--like polynomial of degree $d \ge 2$, that preserves $K_S^\pm$, i.e., $T(K_S^\pm)=K_S^\pm$, then there exists $$C(z_1,z_2, \ldots,z_k)= (\delta_- z_1,\hdots,\delta_- z_{k-\nu},\delta_+z_{k-\nu+1},\hdots,\delta_+ z_k) \text{ with }|\delta_\pm|=1$$ such that 
		$$ \cal{T} \circ \cal{S}=C \circ \cal{S}\circ \cal{T}=\cal{S}\circ \cal{T}\circ C$$
		where $\cal{S}=S^{m}$ and $\cal{T}=T^{m}$ with $m=\text{lcm}(\nu,k-\nu)$.
	\end{thm}
\no	Though, in spirit the proof of Theorem \ref{rigidity of shifts} in the present paper is similar to the proof of Theorem 1.1 in \cite{rigidity}, they deviate significantly from each other. That the Green functions $G_S^\pm$ associated to a shift--like map $S$ is not  necessarily pluriharmonic in $\mathbb{C}^k\setminus K_S^\pm$, is the main reason for this deviation. 
For a single H\'{e}non map $H$, the Green functions $G_H^\pm$ are pluriharmonic in $\mathbb{C}^2\setminus K_H^\pm$ which in turn, gives that the pluricomplex Green functions for the non-escaping sets $K_H^\pm$ are the Green functions $G_H^\pm$. Hence, if we start with a pair of H\'{e}non maps for which the non-escaping sets coincide, then their Green functions also coincide. Since it is not clear whether the pluricomplex Green functions for $K_S^\pm$ are $G_S^\pm$ or not, starting with two shift--like maps with identical non-escaping sets, it is not possible to conclude directly that they have the same Green functions. However, we succeed to show that the two shift--like maps having identical non--escaping sets indeed have identical Green functions using different techniques from \cite{BT} and \cite{DS}.
	
	\medskip 
	\no 
	Another class of maps which we consider in this note, are the skew products of H\'{e}non maps in $\mathbb{C}^3$ of the  form:
	\begin{equation}\label{Dskew-henon}
	H(\lambda,x,y)=(c\lambda, H_\lambda(x,y))
	\end{equation}
	for $(\lambda,x,y)\in \mathbb{C}^3$ with $0\neq c\in \mathbb{C}$. For each $\lambda\in \mathbb{C}$, 
	\begin{equation}\label{desH}
	H_\lambda=H_{m,\lambda} \circ H_{{(m-1)},\lambda}\circ \cdots \circ H_{1,\lambda}
	\end{equation}
	where $H_{j,\lambda}(x,y)=(y,p_{j,\lambda}(y)-\delta_{j} x)$ with $p_{j,\lambda}$  polynomial of degree $d_j\geq 2$ having highest degree coefficient $c_j\neq 0$ and $\delta_j\neq 0$ for $1\leq j \leq m$. Further,  $c_j$'s, $\delta_j$'s and $d_j$'s are independent of $\lambda$.  Further, the coefficients of the polynomial $p_{j,\lambda}$ vary continuously with $\lambda$ for $1\leq j \leq m$.
	 This class of maps first appeared in \cite{FW} in connection to classification of quadratic polynomial automorphisms in $\mathbb{C}^3$ and its dynamics was studied in \cite{FC}.  To address the {\it{rigidity}} property of the Julia sets of these maps,  we first study their dynamics  generalizing the techniques developed in \cite{FC}. 
	
	\medskip 
	\no 
	For each $\lambda\in \mathbb{C}$, let
	$$
	H_\lambda(x,y)=({(H_\lambda)}_1(x,y),{(H_\lambda)}_2(x,y))
	$$
	where the degree of ${(H_\lambda)}_1$ is strictly less than that of ${(H_\lambda)}_2$ and in fact, $\deg {(H_\lambda)}_1$ is  ${d}/{d_m}$ when regarded as a polynomial in $x$ and $y$. Further, let 
	\begin{equation}\label{ABjk}
	x_1^\lambda={(H_\lambda)}_1(x,y)=\sum_{j+k=0}^{{d}/{d_m}} A_{jk}(\lambda) x^j y^k  \text{ and } y_1^\lambda={(H_\lambda)}_2(x,y)=\sum_{j+k=0}^{d} B_{jk}(\lambda) x^j y^k .
	\end{equation}
	In particular,
	\begin{equation*}\label{H1}
	{(H_\lambda)}_2(x,y)=c_H y^d+q_\lambda(x,y)
	\end{equation*}
	where 
	\[
	c_H = \prod_{j=1}^m {c_j}^{d_{j+1}\cdots d_m}
	\]
	with the convention that $d_{j+1}\cdots d_m=1$ when $j=m$, $d=d_m \cdots  d_1$ and $q_\lambda$  a polynomial in $x$ and $y$ of degree strictly less than $d$, for each $\lambda\in \mathbb{C}$. 
	Now  note that
	\begin{equation*}
	H^{-1}(\lambda,x,y)=(c^{-1}\lambda, H_{c^{-1}\lambda}^{-1}(x,y))
	\end{equation*}
	and 
	$$
	H_\lambda^{-1}(x,y)=H_{1,\lambda}^{-1} \circ H_{2,\lambda}^{-1}\circ \cdots \circ H_{m,\lambda}^{-1}(x,y)
	$$
   for each $\lambda\in \mathbb{C}$ and	for all $(x,y)\in \mathbb{C}^2$. Let  
	\begin{equation}\label{abjk}
	\tilde{x}_1^{\lambda}={(H_\lambda^{-1})}_1(x,y)=\sum_{j+k=0}^{d} A_{jk}'(\lambda) x^j y^k  \text{ and } \tilde{y}_1^\lambda={(H_\lambda^{-1})}_2(x,y)=\sum_{j+k=0}^{{d}/{d_1}} B_{jk}'(\lambda) x^j y^k .
	\end{equation}
	In particular,
	\begin{equation*}\label{H-1}
	{(H_\lambda^{-1})}_1(x,y)=c_H' x^d+q_\lambda'(x,y)
	\end{equation*}
	where 
	\[
	c_H'=\prod_{j=1}^m {\left({c_j}\delta_j^{-1}\right)}^{d_{j-1}\cdots d_1} 
	\]
	with the convention that $d_{j-1}\cdots d_1=1$ when $j=1$ and $q_\lambda'$ a polynomial in $x$ and $y$ with degree strictly less than $d$. Further, the maps $\lambda\mapsto A_{jk}(\lambda)$, $\lambda\mapsto B_{jk}(\lambda)$,  $\lambda\mapsto A_{jk}'(\lambda)$ and $\lambda\mapsto B_{jk}'(\lambda)$  are assumed to be continuous in $\mathbb{C}$.
	
	\medskip 
	\no 
	Let $\deg (B_{jk})= l_{jk}$ for $0\leq j+k \leq d$. Define 
	\begin{equation}
	\tilde{d}= \deg H= \max \{ (j+k)+l_{jk}: 0\leq (j+k) \leq d \}.
	\end{equation}
	Consequently,
	\[
	\deg (H^n)= \tilde{d} d^{n-1}
	\]
	for all $n\geq 1$. Further, we fix the following notation:
	\[
	x_n^\lambda={(H_{c^{n-1}\lambda}\circ \cdots \circ H_\lambda) }_1(x,y) \text{ and } y_n^\lambda={(H_{c^{n-1}\lambda}\circ \cdots \circ H_\lambda)}_2(x,y)
	\]
	for all $n\geq 1$ and for each $\lambda\in \mathbb{C}$. Similarly, fix 
	\[
	\tilde{x}_n^\lambda={\left(H_{c^{-n}\lambda}^{-1}\circ \cdots \circ H_{c^{-1}\lambda}^{-1}\right) }_1(x,y) 
\text{ and } \tilde{y}_n^\lambda={\left(H_{c^{-n}\lambda}^{-1}\circ \cdots \circ H_{c^{-1}\lambda}^{-1}\right) }_2(x,y)
	\]
for all $n\geq 1$ and for each $\lambda\in \mathbb{C}$.	

	\medskip 
	\no 
	In Section 3, we study the dynamics of the map of the form (\ref{Dskew-henon}). Depending on modulus of $c$, we choose the sets $V_R^+$ and $V_R^-$ for sufficiently large $R$.   For $\lvert c\rvert >1$,  we define
	\begin{align*}
	V_R^+=\left\{(\lambda,x,y)\in \mathbb{C}^3: \lvert y\rvert > \max \{R, \lvert x \rvert, {\lvert \lambda \rvert}^{\tilde{d}+1}\right\}\\
	V_R^{-}=\left\{(\lambda,x,y)\in \mathbb{C}^3: \lvert x\rvert > \max \{R, \lvert y \rvert\}, \lvert \lambda \rvert <1\right\}.
	\end{align*} 
	For $\lvert c\rvert <1$, the role of $V_R^+$ and $V_R^-$ gets interchanged. For $\lvert c\rvert =1$,  set
	\begin{align*}
	V_R^+=\left \{(\lambda,x,y): \lvert y \rvert > \max \{R, \lvert x \rvert, {\lvert \lambda\rvert}^{\tilde{d}+1}\}\right \}\\
	V_R^-=\left \{(\lambda,x,y): \lvert x \rvert > \max \{R, \lvert y \rvert, {\lvert \lambda\rvert}^{\tilde{d}+1}\}\right \}.
	\end{align*}
	The sets $V_R^\pm$ help to localize the dynamics of the map $H$.
	Define 
	\begin{equation}
	U_H^\pm=\bigcup_{n\geq 1}H^{\mp n}(V_R^\pm) \text{ and } K_H^\pm=\mathbb{C}^3\setminus U_H^\pm.
	\end{equation}
    Further, define the Julia sets (forward and backward) $J_H^\pm=\partial K_H^\pm$. Clearly the sets $K_H^\pm$ and $J_H^\pm$ are invariant under the map $H$.  
	For each $n\geq 1$, define 
	\begin{equation}\label{Green def}
	G_{n, H}^\pm(\lambda,x,y)= \frac{1}{\tilde{d} d^{n-1}} \log^+ \lVert H^{\pm n}(\lambda,x,y)\rVert
	\end{equation}
	for $(\lambda,x,y)\in \mathbb{C}^3$. The sequence of functions $G_{n,H}^\pm$ converge to the functions $G_H^\pm$ uniformly on compacts.  The function $G_H^\pm$ is plurisubharmonic in $\mathbb{C}^3$ and pluriharmonic in $\mathbb{C}^3\setminus K_H^\pm$ satisfying the natural functorial property. Further, $G_H^\pm>0$ in $\mathbb{C}^3\setminus K_H^\pm$  and they vanish precisely on $K_H^\pm$. Properties of Green functions $G_H^\pm$ has been recorded in Theorem \ref{Green_Skew}.
	
	\medskip 
	\no 
	As in the case of single H\'{e}non map, the sets $K_H^\pm$ are not necessarily the collection of the points in $\mathbb{C}^3$ having bounded orbits. For example,
	in case $\lvert c \rvert>1$, for $(\lambda,x,y) \in K_H^+ $ with $\lambda\neq 0$,  $\lvert H^n(\lambda,x,y)\rvert \rightarrow \infty$ as $n\rightarrow \infty$.  If $K_H^+\subseteq \{\lambda=0\}$, then since $G_H^+$ is continuous in the whole $\mathbb{C}^3$ and $G_H^+$ is pluriharmonic away from $K_H^+$ , by removable singularity theorem $G_H^+$ can be extended as pluriharmonic function in whole $\mathbb{C}^3$. Now $G_H^+>0$ in $U_H^+$. Thus $G_H^+$ is identically constant in $\mathbb{C}^3$. This shows that, in this case $K_H^+$ always contain a point in $\mathbb{C}^3$ with unbounded orbit. In Lemmas \ref{estK+} and \ref{estK-}, we give an estimate of the growth rate of a point in $K_H^\pm $.

	\medskip 
	\no 
	We now state the {\it{rigidity}} theorem for skew products of H\'{e}non maps fibered over non-compact parameter space $\mathbb{C}$.
	\begin{thm}\label{skew-ncompact}
		Let $H$ and $F$ be two skew products of H\'{e}non maps in $\mathbb{C}^3$. Further, assume that Julia sets of $H$ and $F$ are the same, i.e.,  $J_{H}^\pm= J_{F}^\pm$. Then, 
		\begin{equation*}
		F\circ H= \gamma \circ H \circ F 
		\end{equation*}
		for some $\gamma:(\lambda, x,y)= (\delta \lambda, \delta_+ x, \delta_- y)$ with $0\neq \delta \in \mathbb{C}$ and  $\lvert \delta_\pm \rvert=1$.
	\end{thm}
	
In the same spirit as above, we can give a {\it{rigidity}}  theorem  for skew-products of H\'{e}non maps fibered over a compact metric space. We conclude this note after discussing the proof of this theorem.

\section{Rigidity of Julia sets of shift--like maps}
\no 
In this section, we first prove a uniqueness result for locally bounded,  plurisubharmonic functions having at most logarithmic growth. It is essentially a modification to Theorem 1 from \cite{BT}. Let
\[
\mathcal{L}=\{u \in L^\infty_{loc}(\mathbb{C}^k)\cap {\cal PSH}(\mbb C^k): u(z)\leq \log(1+\lvert z \rvert)+M \}
\] 
\begin{lem} \label{diff}
	Let $u,v \in \mathcal{L}$ such that	  
	\begin{equation}\label{given}
	{(dd^c u)}^{k-p}={(dd^c v)}^{k-p}
	\end{equation}
	for some $1\leq p\leq (k-1)$, then $u$ and $v$ differ by a constant in $\mathbb{C}^k$.
\end{lem}
\begin{proof}
	Let
	\[
	\rho=u-v
	\]
	and let $L$ be any strongly plurisubharmonic function in $\mathbb{C}^k$. We will show that 
	\[
	d\rho \wedge d^c \rho \wedge {(dd^c L)}^{k-1}=0
	\]
	which in turn gives that $d\rho =0$ almost everywhere and hence $\rho$ is a constant. 
	
	\medskip\no 
	We prove by induction that for $q=0,\ldots, (k-p-1)$
	\begin{equation}\label{inductionL}
	d\rho \wedge d^c \rho \wedge  {(dd^c L)}^{q+p} \wedge {(dd^c u)}^i \wedge {(dd^c v)}^j=0
	\end{equation} 
	on $\mathbb{C}^k$ where $i+j=k-q-p-1$.
	
	\medskip\no We prove (\ref{inductionL}) for $q=0$ first i.e., we prove that 
	\begin{equation} \label{ind fst}
	d\rho \wedge d^c \rho \wedge {(dd^c L)}^p \wedge {(dd^c u)}^i  \wedge {(dd^c v)}^j=0
	\end{equation}
	for $i+j=k-p-1$. Let 
	\[
	\Theta={(dd^c u)}^{k-p-1}+ {(dd^c u)}^{k-p-2} \wedge dd^c v + \cdots +{(dd^c v)}^{k-p-1}.
	\]
	Using (\ref{given}), we have that 
	\begin{equation}
	dd^c \rho \wedge \Theta = 0.
	\end{equation}
	Let 
	\[
	T=\rho d^c \rho \wedge \Theta \wedge {(dd^c L)}^p,
	\]
	then 
	\[
	dT= d\rho \wedge d^c \rho \wedge \Theta \wedge  {(dd^c L)}^p.
	\]
	is an exact positive $(k,k)-$current and thus $dT=0$ which in turn gives (\ref{ind fst}).
	
	\medskip 
	\no 
	Suppose inductively, we have
	\begin{equation} \label{ind}
	d\rho \wedge d^c \rho \wedge {(dd^c L)}^{p+q} \wedge {(dd^c u)}^i  \wedge {(dd^c v)}^j=0
	\end{equation}
	with $i+j=k-p-q-1$. We will be done if we prove that 
	\begin{equation}\label{ind step} 
	d\rho \wedge d^c \rho \wedge {(dd^c L)}^{p+q+1} \wedge {(dd^c u)}^i  \wedge {(dd^c v)}^j=0
	\end{equation}
	with $i+j=k-p-q-2$. Using Schwarz inequality (\ref{ind}) gives that
	\begin{equation} \label{Sch}
	d\rho \wedge {(dd^c L)}^{p+q} \wedge {(dd^c u)}^i  \wedge {(dd^c v)}^j=0
	\end{equation}
	and 
	\begin{equation*} 
	d^c \rho \wedge {(dd^c L)}^{p+q} \wedge {(dd^c u)}^i  \wedge {(dd^c v)}^j=0.
	\end{equation*}
	Let
	\begin{equation*} 
	T=d\rho \wedge d^c \rho \wedge {(dd^c L)}^{p+q} \wedge {(dd^c u)}^i  \wedge {(dd^c v)}^j
	\end{equation*}
	with $i+j=k-p-q-2$. Let 
	\[
	U=d^c L \wedge T.
	\]
	Note that by using (\ref{Sch}), it follows that
	\[
	dT= d\rho \wedge dd^c \rho \wedge {(dd^c L)}^{p+q} \wedge {(dd^c u)}^i  \wedge {(dd^c v)}^j=0
	\]
	Hence,
	\[
	dU=dd^c L \wedge T.
	\]
	As before, $dU$ is an exact, positive, $(k,k)$-current and thus 
	\[
	dU=0.
	\]  
	This proves (\ref{ind step}). Thus we get that $u$ and $v$ differ by a constant in $\mathbb{C}^k$. 
\end{proof}

\no 

\subsection{Proof of Theorem \ref{rigidity of shifts}} 
Recall that a shift--like map $S$ of type $\nu$ ( $1\le \nu \le k-1$) is a map of the form:
\begin{align}\label{Shift S}
 S(z_1,\hdots,z_k)=(z_2,\hdots,z_k,az_1+p(z_{k-\nu+1})
 \end{align} where $a \in \mbb C^*$ and $p$ is a polynomial of degree $d \ge 2.$ From Proposition 2.4 in \cite{mypaper} the $\nu(k-\nu)-$th iterate of a shift--like map $S$ is regular. The positive and negative Green functions of $S$ are defined as:
\[ G_S^+(z)=\lim_{n \to \infty}\frac{\log^+\|S^{\nu n}(z)\|}{d^n} \text{ and }G_S^-(z)=\lim_{n \to \infty}\frac{\log^+\|S^{-(k-\nu)n}(z)\|}{d^n}.\]
Further from Theorem 8.7 in \cite{DS},
\[\mu_S^+=\Big(\frac{1}{2\pi}dd^c G_S^+\big)^{\nu} \text{ and }\mu_S^-=\Big(\frac{1}{2\pi}dd^c G_S^-\Big)^{k-\nu}\] 
are unique currents of mass $1$ supported in $K_S^+$ and $K_S^-$ respectively.

\medskip\no \textit{Step 1: } $T$ is a $\nu-$shift. Also $G_S^\pm=G_T^\pm$ and $K_S^\pm=K_T^\pm.$

\medskip\no Let $T$ be a $\nu_1-$shift such that $1 \le \nu_1 \neq \nu \le k-1.$ Recall from \cite{BP}, the filtration properties of the maps $S$ and $T$. For $R>0$ we define the sets
\[ V_{R,S}^+=\bigcup_{i=k-\nu+1}^k V^i_R \text{ and } V_{R,S}^{-}=\bigcup_{i=1}^{k-\nu} V^i_R\] and 
\[ V_{R,T}^+=\bigcup_{i=k-\nu_1+1}^k V^i_R \text{ and } V_{R,T}^{-}=\bigcup_{i=1}^{k-\nu_1} V^i_R.\]
Also, $$K_S^\pm \subset V_R \cup V_{R,S}^\mp , \text{ and } K_T^\pm\subset V_R \cup V_{R,T}^\mp $$ for a sufficiently large $R > 0.$ 

\medskip\no Let $\nu_1 < \nu.$ Then $V_{R,T}^+ \subset V_{R,S}^+$ and $V_{R,S}^- \subset V_{R,T}^-.$ Pick a $z \in K_S^+.$ By assumption $$T^{n}(z) \in K_S^+ \subset V_R \cup V_{R,S}^-\subset V_R \cup V_{R,T}^-,$$ i.e., $z \in K_T^+.$ Thus $K_S^+ \subset K_T^+.$

\medskip\no Let $\ov{K_T^\pm}=K_T^\pm \cup I^\pm_T \subset \mbb P^k$ and $\ov{K_S^\pm}=K_S^\pm \cup I^\pm_S \subset \mbb P^k$ where $I^\pm_S$ and $I^\pm_T$ are indeterminacy sets of $S$ and $T$ respectively. Since $\nu_1 < \nu$, by Proposition 2.1 in \cite{mypaper} it follows that $I_S^+ \subset I_T^+$ and $I_T^- \subset I_S^-.$  Now from Theorem 8.7 in \cite{DS} $\mu_T^+$ is the unique $(\nu_1,\nu_1)-$closed current supported in $\ov{K_T^+}$ and $\mu_S^+$ is the unique $(\nu,\nu)-$closed current of mass 1 supported in $\ov{K_S^+}.$ Further from Corollary 8.5 in \cite{DS}
$\tilde{\mu}_S=(\frac{1}{2\pi}dd^c G_S^+)^{\nu_1}$ is a $(\nu_1,\nu_1)-$closed current of mass $1$ supported in $\ov{K_S^+}$, hence in $\ov{K_T^+}.$ Since $(\nu_1,\nu_1)-$closed current of mass 1 supported in $\ov{K_T^+}$ is unique, 
\[ \tilde{\mu}_S=\mu_T^+, \text{ i.e., } \Big(\frac{1}{2 \pi}dd^c G_S^+\Big)^{\nu_1}=\Big(\frac{1}{2 \pi}dd^c G_T^+\Big)^{\nu_1}.\]
Applying Lemma \ref{diff}, we have that $G_S^+=G_T^+$ which in turn gives $K_S^+=K_T^+.$ As $K_T^+ \subset V_R \cup V_{R,S}^-$ it follows that $K_T^+ \cap V_{R}^{k-\nu_1}=\emptyset.$ But from Proposition 2.1 in \cite{mypaper} and Proposition 8.3 in \cite{DS}, it follows that $$P_0=[\underbrace{0:\cdots:0:1:0\cdots:0}_{k-\nu_1-\text{th position}}] \in I_T^+ \subset \ov{K_T^+}, \text{ i.e., } K_T^+ \cap V_{R}^{k-\nu_1} \neq \emptyset.$$ Hence $\nu_1 \not \le \nu.$ A similar argument applied for $T^{-1}$ on $K_S^-$ and $K_T^-$ gives $\nu \not \le \nu_1.$ Thus $\nu_1=\nu.$

\medskip\no Assuming $\nu_1=\nu$, exactly the same arguments as above gives $K_S^+ \subset K_T^+$ and $K_S^- \subset K_T^-.$ Since $\mu_T^\pm$ and $\mu_S^\pm$ are unique currents of mass 1 supported in $K_T^\pm$ and $K_S^\pm$ respectively,
\[ \mu_S^+=\mu_T^+ \text{ and } \mu_S^-=\mu_T^-,\]
i.e., \[ (dd^c G_S^+)^\nu=(dd^c G_T^+)^\nu \text{ and } (dd^c G_S^-)^{k-\nu}=(dd^c G_T^-)^{k-\nu}.\]
Hence by Lemma \ref{diff}, $G_S^\pm=G_T^\pm$ and $K_S^\pm=K_T^\pm.$ 

\medskip\no Similar to the form of $S$ in (\ref{Shift S}) let the form of $T$ be 
\[ T(z_1,\hdots,z_k)=(z_2,\hdots,z_k,bz_1+q(z_{k-\nu+1}))\] where $b \in \mbb C^*$ and $q$ a polynomial of degree $d_q  \ge 2.$ Let 
\[ p(z)=\sum_{i=0}^{d_p} c_i z^i \text{ and } q(z)=\sum_{i=0}^{d_q} c'_i z^i.\]
For $\ep>0$ we define the modified sectors $V_i^R(\ep)$, $1 \le i \le k$ as:
\[ V_i^R(\ep)=\Big\{ z \in \mbb C^k \setminus V_R: |z_i|>\max\{|z_j|+\ep, R: 1 \le j \neq i \le k\} \Big\}.\]
\begin{figure}[h!]
	\includegraphics[scale=0.4]{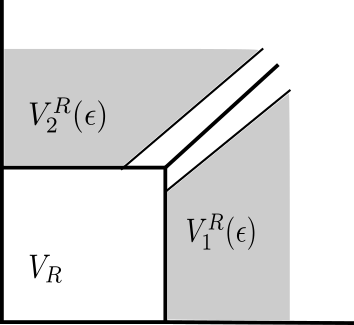}
	\caption{The sets $V_1^R(\ep)$ and $V_2^R(\ep)$ in $\mbb C^2$}
\end{figure}
\textit{Step 2: }There exist modified sectors such that they are invariant under appropriate iterates of $S$ and $T$ or $S^{-1}$ and $T^{-1}.$ 
\begin{lem}\label{Step 3} 
	Corresponding to the shift--like maps $S$ and $T$ (of type $\nu$) there exist $R_0>0$ and $\ep_0>0$ ($\ep_0 \ll R_0$) such that 
	\begin{itemize}
		\item [(i)] For every $k-\nu+1 \le i \le k$,
		$$S^\nu\big (V_i^{R}(\ep_0) \big )\subset V_i^{R}(\ep_0)\text{ and }T^\nu \big (V_i^{R}(\ep_0)\big ) \subset V_i^{R}(\ep_0),$$ whenever $R > R_0.$
		\item [(ii)] For every $1 \le j \le k-\nu$,
		$$S^{-(k-\nu)}\big(V_j^{R}(\ep_0)\big) \subset V_j^{R}(\ep_0)\text{ and }T^{-(k-\nu)}\big(V_j^{R}(\ep_0)\big) \subset V_j^{R}(\ep_0),$$ whenever $R > R_0.$
	\end{itemize}	
\end{lem}
\begin{proof}
	Note that
	\[ S^\nu(z_1,\hdots,z_k)=(z_{\nu+1},\hdots,z_k,az_1+p(z_{k-\nu+1}),\hdots,az_\nu+p(z_k)),\]
	\[S^{-(k-\nu)}\cord= \big(a^{-1}\{z_{\nu+1}-p(z_1)\},\hdots, a^{-1}\{z_k-p(z_{k-\nu})\}, z_1,\hdots,z_{\nu}\big),\]
	and
	\[ T^\nu(z_1,\hdots,z_k)=(z_{\nu+1},\hdots,z_k,bz_1+q(z_{k-\nu+1}),\hdots,bz_\nu+q(z_k)),\]
	\[T^{-(k-\nu)}\cord = \big(b^{-1}\{z_{\nu+1}-q(z_1)\},\hdots, b^{-1}\{z_k-q(z_{k-\nu})\}, z_1,\hdots,z_{\nu}\big).\]
	\textit{Claim: } There exists $R_0 \gg 1$ and $\ep_0>0$ (sufficiently smaller than $R_0$) such that if $z \in D(0;R)$ and $w \in D(0;R+\ep_0)^c$ for $R \ge R_0$ then
	\[ |p(w)|-|p(z)| > 2 R+\ep_0  \text{ and } |q(w)|-|q(z)| > 2 R+\ep_0. \]
	Recall that $p(z)=c_{d_p}z^{d_p}+\cdots+c_0$ and $q(z)=c'_{d_q}z^{d_q}+\cdots+c'_0.$ For a given $\ep>0$, there exists $R(\ep) \gg 1$ such that
	$$(|c_{d_p}|-\ep)|z|^{d_p} \le |p(z)| \le (|c_{d_p}|+\ep)|z|^{d_p},$$
	and
	$$(|c'_{d_q}|-\ep)|z|^{d_q} \le |q(z)| \le (|c'_{d_q}|+\ep)|z|^{d_q},$$
	whenever $|z| \ge R(\ep).$
	Further, we modify $R_0>R(\ep)$ and choose $\ep_0$ (sufficiently large) such that if $z \in D(0;R)$ and $w \in D(0;R+\ep_0)^c$ where $R \ge R_0$, then 
	\[ |w|^{d}-|z|^{d} \ge (R+\ep_0)^d-R^d> 2RM\max\big\{(|c_{d_p}|-\ep)^{-1},(|c'_{d_q}|-\ep)^{-1}\big\}\]
	where  $d=\min(d_p,d_q)$ and $M=\max\{|a|,|b|,1\}.$ Note that 
	\[ |w|^{d_1}-|z|^{d_1} \ge  |w|^{d}-|z|^{d} \text{ and } |w|^{d_2}-|z|^{d_2} \ge  |w|^{d}-|z|^{d} .\]
	Further we modify the choice of $R_0$, such that $ 2 \ep R_0^{d}> M \ep_0.$
	Then for $z \in D(0;R) \setminus D(0;R_0)$ and $w \in D(0;R+\ep_0)^c$, where $R>R_0$, it follows that
	\begin{align*}
	|p(w)|-|p(z)|&>(|c_{d_p}|-\ep)|w|^{d_p}-(|c_{d_p}|+\ep)|z|^{d_p}\\
	&=(|c_{d_p}|-\ep)(|w|^{d_p}-|z|^{d_p})+2 \ep|z|^{d_p}\\
	&>M(2R+\ep_0) \ge 2R+\ep_0.
	\end{align*}
	and
	\begin{align*}
	|q(w)|-|q(z)|&>(|c'_{d_q}|-\ep)|w|^{d_q}-(|c'_{d_q}|+\ep)|z|^{d_q}\\
	&=(|c'_{d_q}|-\ep)(|w|^{d_q}-|z|^{d_q})+2 \ep|z|^{d_q}\\
	&>M(2R+\ep_0) \ge 2R+\ep_0
	\end{align*}
	Now by maximum modulus principle, for each $z \in D(0;R)$ and for each $w \in D(0; R+\ep_0)^c$, where $R \ge R_0$, we have
	\begin{align}\label{p}
	|p(w)|-|p(z)|>M(2R+\ep_0) \text{ and }|q(w)|-|q(z)|>M(2R+\ep_0)
	\end{align}
	and hence the claim follows.
	
	\medskip\no 
	Now let $z \in V_{i_0}^{R}(\ep_0)$ where $k-\nu+1 \le i_0 \le k$ with $R >R_0.$ Let $|z_{i_0}|=\tilde{R}+\ep_0$, then clearly $|z_i| \le \tilde{R}$ for every $1 \le i \neq i_0\le k$ where $\tilde{R}>R_0.$ Observe that 
	\[ \pi_{i} S^\nu(z)=az_{i-k-\nu}+p(z_{i}) \text{ and }\pi_{i} T^\nu(z)=bz_{i-k-\nu}+q(z_{i})\] for $k-\nu+1 \le i \le k.$
	
	\medskip\no 
	\textit{Claim: }For every $z \in V_{i_0}^{R}(\ep_0)$, $|\pi_{i_0} S^\nu(z)|>|\pi_{i} S^\nu(z)|+\ep_0$, for every $1 \le i \le k$ and $ i \neq i_0$ where $R >R_0.$
	
	\medskip\no 
	By (\ref{p}) and by definition of $M$ it follows that if $k-\nu+1 \le i \le k$ and $i \ne i_0,$ then
	\begin{align*}
	&|p(z_{i_0})|-|p(z_i)|>2|a|\tilde{R}+M\ep_0\ge |a|(|z_{i-k+\nu}|+|z_{i_0-k+\nu}|)+\ep_0\\
	&|p(z_{i_0})|-|a||z_{i_0-k+\nu}|>|p(z_i)|+|a||z_{i-k+\nu}|+\ep_0.
	\end{align*}
	But \[|az_{i_0-k+\nu}+p(z_{i_0}|\ge |p(z_{i_0})|-|a||z_{i_0-k+\nu}| \text{ and }|p(z_i)|+|a||z_{i-k+\nu}|\ge |az_{i-k+\nu}+p(z_i)|.\]
	Hence
	\[|\pi_{i_0} S^\nu(z)|>|\pi_{i} S^\nu(z)|+\ep_0\] for $k-\nu+1 \le i \le k.$ Thus the claim follows.
	
	\medskip\no Now $R_0$ is so chosen, such that $S(V_{R}^+) \subset V_{R}^+$ for every $R>R_0$, i.e., $S^\nu\big (V_{i}^R(\ep_0)\big ) \subset V_{i}^R(\ep_0)$ for $k-\nu+1 \le i \le k.$ A similar argument will work for $T$ and this completes the proof of part (i).
	
	\medskip\no 
	Now let $z \in V_{i_0}^{R}(\ep_0)$ and $1 \le j_0 \le k-\nu$ where $R>R_0.$ Let $|z_{j_0}|=\tilde{R}+\ep_0$, then clearly $|z_i| \le \tilde{R}$ for every $1 \le i \neq j_0\le k$ and $\tilde{R}>R_0.$ Observe that 
	\[ \pi_{j} S^{-(k-\nu)}(z)=a^{-1}(z_{\nu+i}+p(z_{i})) \text{ and }\pi_{j} T^{-(k-\nu)}(z)=b^{-1}(z_{\nu+i}+q(z_{i}))\] for $1 \le j \le k-\nu.$
	
	\medskip\no 
	\textit{Claim: }$|\pi_{j_0} S^{-(k-\nu)}(z)|>|\pi_{j} S^{-(k-\nu)}(z)|+\ep_0$ for every $1 \le j \le k-\nu$ and $ j \neq j_0.$
	
	\medskip\no 
	By (\ref{p}) and definition of $M$ it follows that if $k-\nu+1 \le j \le k$ and $j \ne j_0$
	\begin{align*}
	&|p(z_{j_0})|-|p(z_j)|>2\tilde{R}+M\ep_0\ge (|z_{\nu+j}|+|z_{\nu+j_0}|)+|a|\ep_0\\
	&|p(z_{j_0})|-|z_{\nu+j_0}|>|p(z_j)|+|z_{\nu+j}|+|a|\ep_0.
	\end{align*}
	But \[|z_{\nu+j_0}-p(z_{j_0})|\ge |p(z_{j_0})|-|z_{\nu+j_0}| \text{ and }|p(z_j)|+|z_{j-k+\nu}|\ge |z_{j+\nu}-p(z_j)|.\]
	Hence
	\[|\pi_{j_0} S^{-(k-\nu)}(z)|>|\pi_{j} S^{-(k-\nu)}(z)|+\ep_0\] for $k-\nu+1 \le j \le k.$ This proves the claim.
	
	\medskip\no 
	Now $R_0$ is so chosen, such that $S^{-1}(V_{R}^-) \subset V_{R}^-$ for every $R>R_0$, i.e., $S^{-(k-\nu)}\big (V_{j}^R(\ep_0)\big) \subset V_{j}^R(\ep_0)$ for $1 \le j \le k-\nu.$ A similar argument will work for $T$ and this completes the proof of part (ii).
\end{proof}
\no \textit{Step 3: } Note that Lemma \ref{Step 3} assures that appropriate modified sectors are invariant under $S^\nu$ or $S^{-(k-\nu)}$. In this step we construct B\"{o}ttcher coordinates in each of this sectors.
\begin{prop}\label{Step 4}
	Let $S$ be a polynomial shift--like map of type $\nu.$ Then there exist $R_S>0$, $\ep_S>0$ ($\ep_S \ll R_S$) and holomorphic functions $\phi_{S,i}^+$ on $V_i^{R_S}(\ep_S)$, $k-\nu+1 \le i \le k$ and $\phi_{S,j}^-$ on $V_j^{R_S}(\ep_S)$, $1 \le j \le k-\nu$ such that
	\[ \phi_{S,i}^+\big (S^\nu(z)\big )=c_{d_p} \big (\phi_{S,i}^+(z)\big )^{d_p} \text{ and } \phi_{S,j}^-\big (S^{-(k-\nu)}(z)\big )=a^{-1}c_{d_p} \big (\phi_{S,j}^-(z)\big )^{d_p}.\]
	Also $\phi_{S,i}^+(z) \sim z_i$ and $\phi_{S,j}^-(z) \sim z_j$ for $z \to \infty.$
\end{prop}
\begin{proof}
	From Lemma \ref{Step 3}, there exists $R_0$ and $\ep_S$ such that for $1 \le i \le k$ and $1 \le j \le k-\nu$
	\[ S^\nu\big (V_i^{R}(\ep_S) \big) \subset V_i^{R}(\ep_S) \text{ and } S^{-(k-\nu)}\big(V_j^{R}(\ep_S)\big) \subset V_j^{R}(\ep_S)\] for every $R>R_0.$ Let $z_{i,n}=\pi_i \circ S^{\nu n}(z)$ for every $k-\nu+1 \le i \le k$ and $d=d_p$. 
	
	\medskip\no \textit{Claim: }The telescoping product 
	\[ z_{i,0}.\frac{z_{i,1}^{\frac{1}{d}}}{z_{i,0}}. \cdots . \frac{z_{i,n+1}^{\frac{1}{d^{n+1}}}}{z_{i,n}^\frac{1}{d^n}}\]
	converges in $V_i^{R}(\ep_S)$ for $R$ sufficiently large whenever $k-\nu+1 \le i \le k$.
	
	\medskip\no Note that
	\begin{align}\label{eqn1}
	\frac{z_{i,n+1}^{\frac{1}{d^{n+1}}}}{z_{i,n}^\frac{1}{d^n}}=\frac{\{az_{i-k+\nu,n}+p(z_{i,n})\}^{\frac{1}{d^{n+1}}}}{z_{i,n}^\frac{1}{d^n}}.
	\end{align} 
	Since $p(z)=c_d z^d+\hdots+c_0$, (\ref{eqn1}) can be written as
	\begin{align}\label{eqn1.1}
	\frac{z_{i,n+1}^{\frac{1}{d^{n+1}}}}{z_{i,n}^\frac{1}{d^n}}=\bigg(\frac{c_d z_{i,n}^d+\tilde{p}(z_{i-k+\nu,n},z_{i,n})}{z_{i,n}^d}\bigg)^{\frac{1}{d^{n+1}}}=c_d^{\frac{1}{d^{n+1}}}\bigg(1 +\frac{\tilde{p}(z_{i-k+\nu,n},z_{i,n})}{c_dz_{i,n}^d}\bigg)^{\frac{1}{d^{n+1}}}.
	\end{align}
	Now if $R$ is chosen sufficiently large and $z \in V_i^R(\ep_S)$
	\[ \bigg|\frac{\tilde{p}(z_{i-k+\nu,n},z_{i,n})}{c_dz_{i,n}^d}\bigg|<1, \text{ i.e., }\bigg(1+\frac{\tilde{p}(z_{i-k+\nu,n},z_{i,n})}{c_dz_{i,n}^d}\bigg) \in \mbb D(1;1) \subset \mbb C. \]
	We take an appropriate branch of logarithm such that $\text{Log} \;c_d$, $\text{Log} \;a^{-1}c_d$  and $\text{Log} \;z$ are well defined for every $z \in \mbb D(1;1)$ (the ball of radius 1 at 1). Hence, the convergence of the telescoping product 
	\[ z_{i,0}.\frac{z_{i,1}^{\frac{1}{d}}}{z_{i,0}}. \cdots . \frac{z_{i,n+1}^{\frac{1}{d^{n+1}}}}{z_{i,n}^\frac{1}{d^n}}\]
	is equivalent to convergence of the series
	\begin{align}\label{eqn2}
	\frac{1}{d}\text{Log} \bigg(\frac{z_{i,1}}{z_{i,0}^d}\bigg)+\hdots +\frac{1}{d^{n+1}}\text{Log} \bigg(\frac{z_{i,n+1}}{z_{i,n}^d}\bigg).
	\end{align} 
	Note that (\ref{eqn2}) converges for every $z \in V_i^R(\ep_s)$, i.e., the function
	\[ \phi_{S,i}^+(z)=c_d^{-\frac{1}{d-1}} \lim_{n \to \infty} z_{i,n}^{\frac{1}{d^n}}=c_d^{-\frac{1}{d-1}}\lim_{n \to \infty}\bigg(z_{i,0}.\frac{z_{i,1}^{\frac{1}{d}}}{z_{i,0}}. \cdots . \frac{z_{i,n+1}^{\frac{1}{d^{n+1}}}}{z_{i,n}^\frac{1}{d^n}}\bigg)\] is well defined in $V_i^R(\ep_S).$ Note that there exists $M>0$ such that if $R>0$ is sufficiently large
	\begin{align}\label{eqn2.1}
	\bigg|\frac{\tilde{p}(z_{i-k+\nu,n},z_{i,n})}{c_dz_{i,n}^d}\bigg| \le \frac{M}{|z_{i,n}|} \le \frac{\widetilde{M}}{|z_{i,0}|^{d^n}}.
	\end{align}
	Hence from (\ref{eqn1.1}) and (\ref{eqn2.1}) it follows that
	\[ \phi^+_{S,i}(z) \sim z_i \text{ as } z \to \infty \text { in }V_i^R(\ep_S).\] 
	Also 
	{\small \begin{align*}
	\phi_{S,i}^+\big(S^\nu(z)\big)=c_d^{-\frac{1}{d-1}} \lim_{n \to \infty} z_{i,{n+1}}^{\frac{1}{d^n}}=c_d^{-\frac{1}{d-1}} \Big(\lim_{n \to \infty} z_{i,{n+1}}^{\frac{1}{d^{n+1}}}\Big)^d=c_d^{-\frac{1}{d-1}}\Big(c_d^\frac{1}{d-1} \phi_{S,i}^+(z)\Big)^d=c_d\big(\phi_{S,i}^+(z)\big)^d.
	\end{align*}
	}
\no Thus the proof follows.
	
	\medskip\no A similar argument gives that there exist holomorphic functions $\phi_{j,S}^-$ on $V_j^R(\ep_s)$ for $1 \le j \le k-\nu$ such that 
	\[\phi_{S,j}^-\big(S^{-(k-\nu)}(z)\big)=a^{-1}c_d \big(\phi_{S,j}^-(z)\big)^d \]
	and 
	\[\phi_{S,j}^-(z) \sim z_j \text{ as } z \to \infty \text { in }V_j^R(\ep_S).\]
\end{proof}
\no Hence by Lemma \ref{Step 3} and Proposition \ref{Step 4} it is possible to choose $R_1$ sufficiently large such that the following are true:
\begin{itemize}
	\item[(i)] There exist holomorphic functions $\phi_{S,i}^+$ on $V_i^{R_1}(\ep_0)$, $k-\nu+1 \le i \le k$ and $\phi_{S,j}^-$ on $V_j^{R_1}(\ep_0)$, $1 \le j \le k-\nu$ such that
	\[ \phi_{S,i}^+\big(S^\nu(z)\big)=c_{d_p} \big(\phi_{S,i}^+(z)\big)^{d_p} \text{ and } \phi_{S,j}^-\big(S^{-(k-\nu)}(z)\big)=a^{-1}c_{d_p} \big(\phi_{S,j}^-(z)\big)^{d_p}.\]
	Also $\phi_{S,i}^+(z) \sim z_i$ and $\phi_{S,j}^-(z) \sim z_j$ for $z \to \infty.$
	\item[(ii)]There exist holomorphic functions $\phi_{T,i}^+$ on $V_i^{R_1}(\ep_1)$, $k-\nu+1 \le i \le k$ and $\phi_{T,j}^-$ on $V_j^{R_1}(\ep_0)$, $1 \le j \le k-\nu$ such that
	\[ \phi_{T,i}^+\big(T^\nu(z)\big)=c'_{d_q} \big(\phi_{T,i}^+(z)\big)^{d_q} \text{ and } \phi_{T,j}^-\big(T^{-(k-\nu)}(z)\big)=b^{-1}c'_{d_q} \big(\phi_{T,j}^-(z)\big)^{d_q}.\]
	Also $\phi_{T,i}^+(z) \sim z_i$ and $\phi_{T,j}^-(z) \sim z_j$ for $z \to \infty.$
\end{itemize}
\medskip\no \textit{Step 4: }The B\"{o}ttcher coordinates of $S^\nu$, $S^{-(k-\nu)}$ and $T^\nu$, $T^{-(k-\nu)}$ are equal in appropriate modified sectors.  
\begin{lem}
For $k-\nu+1 \le i \le k$, on $V_i^{R_1}(\ep_0)$, $\phi_{S,i}^+=\phi_{T,i}^+$ and for $1 \le j \le k-\nu$, on $V_j^{R_1}(\ep_0)$, $\phi_{S,j}^-=\phi_{T,j}^-.$
\end{lem}
\begin{proof}
Let $R_0$ and $\ep_0$ be as chosen in Lemma \ref{Step 3}. Then for every $R>R_0$:
\begin{itemize}
	\item [(i)] For $k-\nu+1 \le i \le k$ and $z \in V_i^{R}(\ep_0)$
	$$G_S^+(z)=\lim_{n \to \infty}\frac{\log|\pi_i \circ S^{\nu n}(z)|}{d_p^n}\text{ and }G_T^+(z)=\lim_{n \to \infty}\frac{\log|\pi_i \circ T^{\nu n}(z)|}{d_q^n}.$$  
	\item [(ii)] For $1 \le j \le k-\nu$ and $z \in V_j^{R}(\ep_0)$
	$$G_S^-(z)=\lim_{n \to \infty}\frac{\log|\pi_j \circ S^{-(k-\nu) n}(z)|}{d_p^n}\text{ and }G_T^-(z)=\lim_{n \to \infty}\frac{\log|\pi_j \circ T^{-(k-\nu)n}(z)|}{d_q^n}.$$ 
\end{itemize}
Now from the proof of Proposition \ref{Step 4}, for $k-\nu+1 \le i \le k$ and $z \in V_i^{R_1}(\ep_0)$
$$G_S^+(z)={\log\Big|c_{d_p}^{\frac{1}{d_p-1}} \phi_{S,i}^+(z)\Big |}\text{ and }G_T^+(z)={\log\Big|{c'}_{d_q}^{\frac{1}{d_q-1}} \phi_{T,i}^+(z)\Big|}$$
and for $1\le j \le k-\nu$ and $z \in V_j^{R_1}(\ep_0)$
$$G_S^-(z)={\log\Big |(a^{-1}c_{d_p})^{\frac{1}{d_p-1}} \phi_{S,i}^-(z)\Big |}\text{ and }G_T^-(z)={\log\Big|(b^{-1}{c'}_{d_q})^{\frac{1}{d_q-1}} \phi_{T,i}^-(z)\Big|}.$$ Hence
$$G_S^+(z)=\log |\phi_{S,i}^+(z)|+\frac{1}{d_p-1}\log|c_p| , \; \; G_T^+(z)=\log |\phi_{T,i}^+(z)|+\frac{1}{d_q-1}\log|c_q'|$$ and
$$G_S^-(z)=\log |\phi_{S,j}^-(z)|+\frac{1}{d_p-1}\log|a^{-1}c_p| , \;\; G_T^-(z)=\log |\phi_{T,i}^-(z)|+\frac{1}{d_q-1}\log|b^{-1}c_q'|.$$
Since $G_S^+=G_T^+$ (by \textit{Step 1}), for $k-\nu+1 \le i \le k$ and for $z \in V_i^{R_1}(\ep_0)$
\begin{align}\label{eqn3}
\log |\phi_{S,i}^+(z)|+\frac{1}{d_p-1}\log|c_p|=\log |\phi_{T,i}^+(z)|+\frac{1}{d_q-1}\log|c'_q|.
\end{align} 
Now by Lemma \ref{Step 4}, both $\phi_{S,i}^+(z),\phi_{T,i}^+(z) $ asymptotic to $z_i$ as $z \to \infty$ in $V_i^{R_1}(\ep_0)$, i.e.,
\begin{align}\label{eqn3.1}
\phi_{S,i}^+=\phi_{T,i}^+ \text{ and } \frac{1}{d_p-1}\log|c_p|=\frac{1}{d_q-1}\log|c'_q|.
\end{align}
A similar argument gives that on $V_j^{R_1}(\ep_1)$, $1 \le j \le k-\nu$
\begin{align}\label{eqn3.2}
\phi_{S,j}^-=\phi_{T,j}^- \text{ and } \frac{1}{d_p-1}\log|a^{-1}c_p|=\frac{1}{d_q-1}\log|b^{-1}c'_q|.
\end{align}
\end{proof}
\no From now onwards we will use $\phi_i^+$, to denote $\phi_{S,i}^+=\phi_{T,i}^+$ on $V_i^{R_1}(\ep_0)$, for $k-\nu+1 \le i \le k$ and $\phi_j^-$, to denote $\phi_{S,j}^-=\phi_{T,j}^-$ on $V_j^{R_1}(\ep_0)$, for $1 \le j \le k-\nu.$

\medskip\no 
\textit{Step 5: } The coordinates $\mathcal{S} \circ \mathcal{T}(z)$ is related to the coordinates of $\mathcal{T} \circ \mathcal{S}(z)$ where $\cal{S}=S^m$, $\cal{T}=T^m$ and $m=lcm(\nu,k-\nu).$

 \begin{lem}\label{step 5}\leavevmode\vspace{-.1\baselineskip}
	   \begin{enumerate}
	  \item[(i)] For every $z \in \mbb C^k$ and for $k-\nu+1 \le i \le k$ there exists $\delta_+$ with $|\delta_+|=1$ such that 
\begin{align}\label{relation1}
 \pi_i \circ \cal{T} \circ \cal{S}(z)=\delta_+\big(\pi_i \circ \cal{S} \circ \cal{T}(z)\big).
 \end{align}
\item[(ii)]For every $z \in \mbb C^k$ and for $1 \le j \le k-\nu$ there exist $\delta^j_-$'s with $|\delta^j_-|=1$ such that 
\begin{align}\label{relation2}
 \pi_j \circ \cal{T} \circ \cal{S}(z)=\delta^j_-\big(\pi_j \circ \cal{S} \circ \cal{T}(z)\big).
 \end{align}
\end{enumerate}
\end{lem}
\begin{proof} Let $m=m_1 \nu=m_2(k-\nu).$ From Lemma \ref{Step 4}, on $V_i^{R_1}(\ep_0)$, $k-\nu+1 \le i \le k$, we have  
\begin{align}\label{eqn4.1}
\phi_i^+\big(\cal{S}(z)\big)=c_p^{m_1}\big(\phi_i^+(z)\big)^{d_p m_1} \text{ and }\phi_i^+\big(\cal{T}(z)\big)={c'}_q^{m_1}\big(\phi_i^+(z)\big)^{d_q m_1}.
\end{align}
Further from (\ref{eqn3.1}) there exists $\delta_1 \in \mbb C$ such that $|\delta_1|=1$ and
\begin{align}\label{eqn4.2}
c_p^{d_q}c'_q={c'}_q^{d_p}c_p\delta_1.
\end{align}
From (\ref{eqn4.1}) thus on $V_i^{R_1}(\ep_0)$, $k-\nu+1 \le i \le k$, we have the following:
\begin{align*}
&\phi_i^+\big(\cal{S} \circ \cal{T}(z)\big)=c_p^{m_1}{c'}_q^{m_1 d_p}\big(\phi_i^+(z)\big)^{m_1^2d_pd_q} \text{ and }\\ &\phi_i^+\big(\cal{T} \circ \cal{S}(z)\big)={c'}_q^{m_1}{c}_p^{m_1 d_q}\big(\phi_i^+(z)\big)^{m_1^2d_pd_q}.
\end{align*}
Hence from (\ref{eqn4.2}) on $V_i^{R_1}(\ep_0)$, $k-\nu+1 \le i \le k$ we have
\begin{align}\label{eqn4.3}
\phi_i^+\big(\cal{T} \circ \cal{S}(z)\big)=\delta_1^{m_1}\phi_i^+\big(\cal{S} \circ \cal{T}(z)\big)=\delta_+\phi_i^+\big(\cal{S} \circ \cal{T}(z)\big)
\end{align}
where $\delta_+=\delta_1^{m_1}.$ From Lemma \ref{Step 4}, if $z \to \infty$ on $V_i^{R_1}(\ep_0)$,
\[\phi_i^+\big(\cal{T} \circ \cal{S}(z)\big) \sim \pi_i \circ \cal{T} \circ \cal{S}(z)\]
and 
\[\phi_i^+\big(\cal{S} \circ \cal{T}(z)\big) \sim \pi_i \circ \cal{S} \circ \cal{T}(z).\] So from (\ref{eqn4.3}),
\begin{align}\label{eqn4.4}
\pi_i \circ \cal{T} \circ \cal{S}(z)-\delta_+\big(\pi_i \circ \cal{S} \circ \cal{T}(z) \big)\sim 0 \text{ as }z \to \infty.
\end{align}
Fix $z_0 \in \mbb C^{k-1}$ and consider the line $L_{z_0}^i=\{(\xi,z_0) \in \mbb C^k: \pi_i(\xi,z_0)=\xi\}.$ Then $(\xi,z_0) \in V_i^{R_1}(\ep_0)$ for $|\xi|$ sufficiently large. Also $(\xi,z_0) \to \infty$ if $\xi \to \infty.$ Thus from (\ref{eqn4.4}) it follows that
\begin{align}\label{eqn4.5}
\pi_i \circ \cal{T} \circ \cal{S}(\xi,z_0)-\delta_+\big(\pi_i \circ \cal{S} \circ \cal{T}(\xi,z_0)\big) \to 0 
\end{align}
as $\xi \to \infty.$ Since (\ref{eqn4.5}) is a polynomial in $\xi$, it has to be identically the zero polynomial, i.e.,
\begin{align}\label{eqn4.6}
\pi_i \circ \cal{T} \circ \cal{S}(\xi,z_0)=\delta_+(\pi_i \circ \cal{S} \circ \cal{T}(\xi,z_0)).
\end{align}
Now (\ref{eqn4.6}) is true for any $z_0 \in \mbb C^{k-1}$, hence (\ref{relation1}) is true. This completes the proof of (i).

\medskip\no 
From Proposition \ref{Step 4}, on $V_j^{R_1}(\ep_0)$, $1 \le j \le k-\nu$ we have 
\begin{align}\label{eqn5.1}
\phi_j^-\big(\cal{S}^{-1}(z)\big)=(a^{-1}c_p)^{m_2}\big(\phi_j^-(z)\big)^{d_p m_2} \text{ and }\phi_j^-\big(\cal{T}^{-1}(z)\big)=(b^{-1}{c'}_q)^{m_2}\big(\phi_j^-(z)\big)^{d_q m_2}.
\end{align}
Therefore for $z \in \cal{S}^{-1}\big(V_j^{R_1}(\ep_0)\big)$,  $1 \le j \le k-\nu$ we have
\[\phi_j^-(z)=(a^{-1}c_p)^{m_2}\big(\phi_j^-\circ \cal{S}(z)\big)^{d_p m_2}\]
and $z \in \cal{T}^{-1}\big(V_j^{R_1}(\ep_0)\big)$,  $1 \le i \le k-\nu$
\[\phi_j^-(z)=(b^{-1}c'_q)^{m_2}\big(\phi_j^-\circ \cal{T}(z)\big)^{d_q m_2}.\]
Let $U_j^1=(\cal{S} \circ \cal{T})^{-1}\big(V_j^{R_1}(\ep_0)\big)$, and $U_j^2=(\cal{T} \circ \cal{S})^{-1}\big(V_j^{R_1}(\ep_0)\big)$,  $1 \le i \le k-\nu.$  Note that $U_j^1 \cap U_j^2 \neq \emptyset$ as both $U_j^1$ and $U_j^2$ are neighbourhood of the point $$[\underbrace{0:\cdots:0:1:0\cdots:0}_{j-\text{th position}}] \in \mbb P^k.$$
Let $\cal{U}_j=U_j^1 \cap U_j^2$ for $1 \le j \le k-\nu.$ Thus for $z \in \cal{U}_j$
\begin{align}\label{eqn5.2}
\phi_j^-(z)={(b^{-1}{c'}_q)}^\nu(a^{-1}c_p)^{d_q\nu}\big(\phi_j^- \circ \cal{S} \circ \cal{T}(z)\big)^{d_qd_p\nu^2}
\end{align} 
and
\begin{align}\label{eqn5.3}
\phi_j^-(z)={(b^{-1}{c'}_q)}^{d_p m_2}(a^{-1}c_p)^{m_2}\big(\phi_j^-\circ \cal{T} \circ \cal{S}(z)\big)^{d_qd_p m_2^2}.
\end{align}
Further from (\ref{eqn3.2}) there exists $\delta_2 \in \mbb C$ such that $|\delta_2|=1$ and
\begin{align}\label{eqn5.4}
(a^{-1}c_p)^{d_q}(b^{-1}c'_q)=(b^{-1}c'_q)^{d_p}(a^{-1}c_p)\delta_2.
\end{align}
Combining (\ref{eqn5.2})--(\ref{eqn5.4}), for $z \in \cal{U}_j$, we have that
\begin{align}
\phi_j^-\big (\cal{T} \circ \cal{S}(z))^{d_qd_p m_2^2}=\delta_2^{m_2}\big (\phi_j^-\circ \cal{S} \circ \cal{T}(z)\big)^{d_qd_p m_2^2}.
\end{align}
For every $1 \le j \le k-\nu$ we take appropriate root $\delta_2^{m_2}$, say $\delta_-^j$ (note that $|\delta_-^j|=1$) on $\cal{U}_j$ we have that 
\begin{align}\label{eqn5.5}
\phi_j^-\big (\cal{T} \circ \cal{S}(z)\big )=\delta_-^j\phi_j^-\big(\cal{S} \circ \cal{T}(z)\big).
\end{align}
Fix a $j$, $1 \le j \le k-\nu$ and $c=(c_1,\hdots,c_{k-1}) \in \mbb C^{k-1}$ . Further for $\xi \in \mbb C$, consider the points in the projective space
\[ P_\xi=[\underbrace{c_1:\cdots:c_{j-1}:\xi:c_{j}\cdots:c_{k-1}:1}_{j-\text{th position}}].\]
Let $\xi_n \in \mbb C$ such that $|\xi_n| \to \infty.$ Then 
\[P_{\xi_n}=[c_1:\cdots:c_{j-1}:\xi:c_{j}\cdots:c_{k-1}:1] \to [\underbrace{0:\cdots:0:1:0\cdots:0}_{j-\text{th position}}] \in \mbb P^k,\] i.e.,
$x_n=(c_1,\hdots,c_{j-1},\xi_n,c_{j}\hdots,c_{k-1}) \in \cal{U}_j$ for $n$ sufficiently large. Note
\[ P_1(\xi)=\pi_j \circ \cal{S} \circ \cal{T}(c_1,\hdots,c_{j-1},\xi,c_{j}\hdots,c_{k-1})\]
and
\[ P_2(\xi)=\pi_j \circ \cal{T} \circ \cal{S}(c_1,\hdots,c_{j-1},\xi,c_{j}\hdots,c_{k-1})\]
are polynomials in $\mbb C.$ Thus $P_1(\xi_n) \to \infty$, $P_2(\xi_n) \to \infty$  as $n \to \infty$, i.e., 
\[ \cal{S} \circ \cal{T}(x_n) \in V_j^{R_1}(\ep_0) \text{ and } \pi_j \circ \cal{S} \circ \cal{T}(x_n) \to \infty\]and
\[ \cal{T} \circ \cal{S}(x_n) \in V_j^{R_1}(\ep_0) \text{ and } \pi_j \circ \cal{T} \circ \cal{S}(x_n) \to \infty\]
as $n \to \infty.$ Therefore from Lemma \ref{Step 4} we have that
\[\phi_j^- \circ \cal{S} \circ \cal{T}(x_n)\sim\pi_j \circ \cal{S} \circ \cal{T}(x_n)\]
and
\[\phi_j^- \circ \cal{T} \circ \cal{S}(x_n)\sim\pi_j \circ \cal{T} \circ \cal{S}(x_n).\] Now from (\ref{eqn5.5}), it follows that
\begin{align*}
\pi_j\big(\cal{T} \circ \cal{S}(x_n)\big)-\delta_-^j\pi_j\big(\cal{S} \circ \cal{T}(x_n)\big) \to 0, \text{ i.e., } P_2(\xi_n)-\delta_-^jP_1(\xi_n) \to 0
\end{align*}
as $n \to \infty.$ Since $P_2 -\delta_-^jP_1$ is polynomial in one variable, the above is not possible unless it is the zero polynomial. Hence for every $\xi \in \mbb C$
\[\pi_j \circ \cal{T} \circ \cal{S}(c_1,\hdots,c_{j-1},\xi,c_{j}\hdots,c_{k-1})=\delta_-^j \pi_j \circ \cal{S} \circ \cal{T}(c_1,\hdots,c_{j-1},\xi,c_{j}\hdots,c_{k-1}).\]
Now the above observation is true for any $c \in \mbb C^{k-1}$, thus (\ref{relation2}) holds. This completes the proof of (ii).
\end{proof}
\no Thus there exists a linear map $ C_1(z_1,\hdots,z_k)=(\tilde{\delta_1}z_1,\hdots,\tilde{\delta_k}z_k)$ such that
\begin{align}\label{eqn6.1}
\cal{T} \circ \cal{S}(z)=C_1 \circ \cal{S} \circ \cal{T}(z)
\end{align}
where $\tilde{\delta_j}=\delta_-^j$ for $1 \le j \le k-\nu$ and $\tilde{\delta_i}=\delta_+$ for $k-\nu+1 \le i \le k.$

\medskip\no \textit{Step 6: } The coordinates of $\cal{T}^{-1} \circ \cal{S}^{-1}$ and $\cal{S}^{-1} \circ \cal{T}^{-1}$ are also related in a similar way.
\begin{lem}\label{step 6} 
There exists a linear map
\[ C_2(z_1,\hdots,z_k)=(\delta_1'z_1,\hdots,\delta_k'z_k)\]
where $\delta'_j=\delta'_-$ for $1 \le j \le k-\nu$ and $\delta'_i={\delta'_+}^i$ for $k-\nu+1 \le i \le k$ such that 
\begin{align}\label{eqn6.2}
 \cal{T}^{-1} \circ \cal{S}^{-1}(z)=C_2 \circ \cal{S}^{-1} \circ \cal{T}^{-1}(z) \text{ i.e., }
\cal{T} \circ \cal{S}(z)= \cal{S} \circ \cal{T}\circ C_2 (z).
\end{align}
\end{lem}
\begin{proof}
The proof is similar to Lemma \ref{step 5} if the role of $\cal{T}$ and $\cal{S}$ are interchanged with the role of $\cal{T}^{-1}$ and $\cal{S}^{-1}$ appropriately.
\end{proof}
\no\textit{Step 7: }Last step is to prove that the maps $C_1$ and $C_2$ are equal.

\medskip\no
From (\ref{eqn6.1}) and (\ref{eqn6.2}) it follows that 
\begin{align}\label{eqn6.3}
\cal{S} \circ \cal{T}= C_1^{-1} \circ \cal{S} \circ \cal{T} \circ C_2 (z).
\end{align}
Let $D_1$, $D_2$ be diagonal matrices defined as follows:
\[ D_1=\text{Diag}({\delta_-^1},\hdots, {\delta_-^{k-\nu}},\delta_+, \hdots, \delta_+) 
\text{ and }D_2=\text{Diag}({\delta'_-},\hdots, \delta'_-,{\delta'}_+^{k-\nu+1}, \hdots, {\delta'}_+^{k}).\]
Now by applying chain rule to (\ref{eqn6.3}) it follows that  
\begin{align}\label{eqn6.4}
D_1 A=AD_2
\end{align}
where $A=D(\cal{S} \circ \cal{T})(0)$. Note that $A$ is invertible and $A$, $D_1$, $D_2$ has the following forms:
{\small
	\begin{align*}
	A=\left[
	\begin{array}{c|c}
	\cal{A}_{k-\nu , k-\nu}  &\cal{B}_{k-\nu ,\nu} \\ \hline
	\cal{C}_{\nu,  k-\nu} &\cal{D}_{\nu , \nu}
	\end{array}\right], \;
	D_1=\left[
	\begin{array}{c|c}
	\cal{D}_1  &0 \\ \hline
	0 &\text{Id}_\nu\delta_+^{-1}
	\end{array}\right], \text{ and }
	D_2=\left[
	\begin{array}{c|c}
	\text{Id}_{k-\nu}\delta'_-  &0 \\ \hline
	0 &\cal{D}_2
	\end{array}\right]
	\end{align*} 
}
where $\cal{D}_1=\text{Diag}({\delta_-^1},\hdots, {\delta_-^{k-\nu}})$ and $\cal{D}_2=\text{Diag}({\delta'}_+^{k-\nu+1}, \hdots, {\delta'}_+^{k}).$ Thus (\ref{eqn6.4}) simplifies as:

\begin{align}\label{eqn6.5}
\left[
\begin{array}{c|c}
\cal{D}_1\cal{A}_{k-\nu , k-\nu}  &\cal{D}_1\cal{B}_{k-\nu ,\nu} \\ \hline
\delta_+\cal{C}_{\nu,  k-\nu} &\delta_+\cal{D}_{\nu , \nu}
\end{array}\right]=\left[
\begin{array}{c|c}
\delta'_-\cal{A}_{k-\nu , k-\nu}  &\cal{B}_{k-\nu ,\nu}\cal{D}_2 \\ \hline
\delta'_-\cal{C}_{\nu,  k-\nu} &\cal{D}_{\nu , \nu}\cal{D}_2
\end{array}\right].
\end{align}
\textit{Case 1:} Suppose $\nu \neq k-\nu$, i.e., either $\nu< k-\nu$ or $k-\nu< \nu.$

\medskip\no Assume $k-\nu<\nu.$ Since $A$ is invertible, the possibilities are
\begin{itemize}
	\item[(i)] $\cal{A}$ and $\cal{D}$ should be invertible.
	\item[(ii)] If $\cal{A}$ or $\cal{D}$ is not invertible, then $\text{Rank}(\cal{B})=k-\nu.$
\end{itemize}
\textit{Subcase (i): } If $\cal{A}$ and $\cal{D}$ are invertible, then $\cal{D}_1=\delta'_-\text{Id}$ and $\cal{D}_2=\delta_+\text{Id}$, i.e., $\delta_-^i=\delta'_-$ for $1 \le i \le k-\nu$ and ${\delta'}_+^i=\delta_+$ for $k-\nu+1 \le i \le k.$ This proves that for $|\delta_\pm|=1$,
$$C_1(z_1,z_2, \hdots,z_k)= C_2(z_1,z_2, \hdots,z_k)=(\delta_- z_1,\hdots,\delta_- z_{k-\nu},\delta_+z_{k-\nu+1}\hdots,\delta_+ z_k).$$ 
\no\textit{Subcase (ii): }Otherwise if $\text{Rank}(\cal{B})=k-\nu$, then any $(k-\nu)-$columns of $\cal{B}$ are linearly independent. From(\ref{eqn6.5}) it follows that,
{\small
	\[
	\begin{pmatrix}
	\delta_-^1 b_{11} & \delta_-^1 b_{12}  &\hdots &\delta_-^1 b_{1\nu} \\
	\delta_-^2 b_{21} & \delta_-^2 b_{22}  &\hdots &\delta_-^2 b_{2\nu} \\
	\vdots &\vdots &\hdots &\vdots\\
	\delta_-^{k-\nu} b_{(k-\nu)1} & \delta_-^{k-\nu} b_{(k-\nu)2}  &\hdots &\delta_-^{k-\nu} b_{(k-\nu)\nu} 
	\end{pmatrix}
	=\begin{pmatrix}
	{\delta'}_+^1 b_{11} & {\delta'}_+^2 b_{12}  &\hdots &{\delta'}_+^\nu b_{1\nu} \\
	{\delta'}_+^1 b_{21} & {\delta'}_+^2 b_{22}  &\hdots &{\delta'}_+^\nu b_{2\nu} \\
	\vdots &\vdots &\hdots &\vdots\\
	{\delta'}_+^1 b_{(k-\nu)1} & {\delta'}_+^2 b_{(k-\nu)2}  &\hdots &{\delta'}_+^{k-\nu} b_{(k-\nu)\nu} 
	\end{pmatrix}.
	\]
}
\no Construct $\widetilde{\cal{B}}$ by choosing any $(k-\nu)-$columns of $\cal{B}.$ Further, choose the corresponding $(k-\nu)$ eigenvalues of $\cal{D}_2$ and denote it by $\widetilde{\cal{D}_2}.$ Then
\[ \cal{D}_1 \widetilde{B}=\widetilde{B} \widetilde{\cal{D}_2}, \text{ i.e., } {\text{eigenvalues of }}\cal{D}_1={\text{ eigenvalues of }}\widetilde{\cal{D}_2}.\]
Now since any $(k-\nu)-$eigenvalues of $\cal{D}_2$ is equal to the eigenvalues of $\cal{D}_1$, we get that all the eigenvalues of $\cal{D}_1$ and $\cal{D}_2$ should be equal. This proves that $\cal{D}_1=\rho \text{Id}_{k-\nu}$ and $\cal{D}_2=\rho \text{Id}_\nu.$
Note that since $\cal{A}$ and $\cal{D}$ are not invertible, $\cal{C}$ should have non-zero elements, i.e., $\delta'_-=\delta_+=\eta.$ Further, $\cal{A}$ and $\cal{D}$ cannot be identically zero matrices, since that would mean $\text{Rank}(A)< k.$ Hence $\eta=\rho.$
This proves that for $|\eta|=1$,
$$C_1(z_1,z_2, \hdots,z_k)= C_2(z_1,z_2, \hdots,z_k)=(\eta z_1,\hdots,\eta z_k)$$  
The similar argument will work if $\nu<k-\nu$ by interchanging the role of rows and columns of $\cal{B}.$

\medskip\no \textit{Case 2: } For $\nu=k-\nu$, i.e., $\text{lcm}(\nu,k-\nu)=\nu$ and $\cal{S}=S^\nu$, $\cal{T}=T^\nu.$

\medskip\no So by chain rule $$A=DG(0)=DS^\nu(T^\nu(0))DT^\nu(0).$$ Also 
{\small
	$$T^\nu(0,\hdots,0)=(0,\hdots,0,q(0),\hdots,q(0)) \text{ and }
	S^\nu \circ T^\nu(0,\hdots,0)=\big(q(0),\hdots,q(0), p(q(0)),\hdots,p(q(0))\big).$$
}
Now 
\[DT^\nu(0)=
\left[
\begin{array}{c|c}
0  & \text{Id}_\nu \\ \hline
b\text{Id}_\nu &q'(0)\text{Id}_\nu
\end{array}\right]
\text{ and }
DS^\nu(T^\nu(0))=
\left[
\begin{array}{c|c}
0  & \text{Id}_\nu \\ \hline
a\text{Id}_\nu &p'(q(0))\text{Id}_\nu
\end{array}\right].\]
Thus 
\[ A=
\left[
\begin{array}{c|c}
b\text{Id}_\nu  & q'(0)\text{Id}_\nu \\ \hline
bp'(q(0))\text{Id}_\nu &(a+q'(0)p'(q(0)))\text{Id}_\nu
\end{array}\right]
.\]
As before from (\ref{eqn6.4}) and (\ref{eqn6.5}), by analyzing all possible cases we again have
$$C_1(z_1,z_2, \hdots,z_k)= C_2(z_1,z_2, \hdots,z_k)=(\delta_- z_1,\hdots,\delta_- z_{\nu},\delta_+z_{\nu+1}\hdots,\delta_+ z_k)$$ where $|\delta_\pm|=1.$ This completes the proof of Theorem \ref{rigidity of shifts}.

\section{Dynamics of skew products of H\'{e}non maps}
In this section, we describe dynamics of the maps of the form (\ref{Dskew-henon}). We consider the following three cases:
\begin{center}
 {\it{Case 1}}: $\lvert c \rvert>1$
\end{center}
Fix a $\delta\in \left(0,1\right)$. For $R>1$, recall that
\begin{equation*}
V_R^+=\left\{(\lambda,x,y)\in \mathbb{C}^3: \lvert y\rvert > \max \{R, \lvert x \rvert, {\lvert \lambda \rvert}^{\tilde{d}+1}\right\}.
\end{equation*}
\begin{lem}\label{est y}
	For any $\delta\in \left(0, h\right)$ where $h=\min\{\lvert c_H\rvert, 1\}$, there exists $R_0=R_0(\delta)>1$ such that
	\begin{equation*}
	H(V_R^+)\subseteq V_R^+
	\end{equation*}
	for all $R\geq R_0$. Further,
	\begin{equation}\label{estimate y}
	{\left(\lvert c_H\rvert-\delta \right)}^{\frac{d^n-1}{d-1}}{\lvert y \rvert}^{d^n}<\lvert y_n^\lambda \rvert< {\left(\lvert c_H\rvert+\delta\right)}^{\frac{d^n-1}{d-1}}{\lvert y \rvert}^{d^n}
	\end{equation}
	for all $(\lambda,x,y)\in V_R^+$.
\end{lem}
\begin{proof}
	We prove (\ref{estimate y}) by induction. Note that
	\begin{equation}\label{y1}
	\left\lvert y_1^\lambda - c_H y^d\right\rvert= \left\lvert q_\lambda(x,y)\right\rvert= \left\lvert \sum_{j+k=0}^{d-1} c_{jk}(\lambda)x^j y^k\right\rvert
	\end{equation}
	where $c_{jk}$'s are polynomials in $\mathbb{C}$ with degree at most $\tilde{d}$. Therefore, there exists $M>1$ such that 
	\begin{equation}\label{estimate pol}
	\left\lvert c_{jk}(\lambda)\right\rvert < M{\lvert \lambda\rvert}_+^{\tilde{d}}
	\end{equation}
	for all $j,k$'s with $0\leq (j+k) \leq (d-1)$ where ${\lvert\lambda\rvert}_+=\max\{\lvert \lambda\rvert,1\}$.  For $(\lambda,x,y)\in V_R^+$, combining (\ref{y1}) and (\ref{estimate pol}), we get  that
	\begin{eqnarray}\label{est q}
	\left\lvert q_\lambda(x,y)\right\rvert &\leq& M{\lvert \lambda\rvert}_+^{\tilde{d}}{\lvert y \rvert}^{d-1}+ M{\lvert \lambda\rvert}_+^{\tilde{d}}{\lvert y \rvert}^{d-2}+\cdots+ M{\lvert \lambda\rvert}_+^{\tilde{d}}\lvert y\rvert+ M{\lvert \lambda\rvert}_+^{\tilde{d}}\nonumber\\
	&\leq& M{\lvert y \rvert}^d \left( {\lvert y\rvert}^{-\frac{1}{\tilde{d}+1}}+ {\lvert y \rvert}^{-1}+\cdots+ {\lvert y\rvert}^{-\frac{\tilde{d}(d-1)+d}{\tilde{d}+1}} \right).
	\end{eqnarray}
	The last inequality follows since $\lvert y\rvert> {\lvert\lambda\rvert}_+^{\tilde{d}+1}$ which in turn gives  ${\lvert \lambda\rvert}_+^{\tilde{d}}< {\lvert y \rvert}^{\frac{\tilde{d}}{\tilde{d}+1}}$.
	Thus it follows from (\ref{est q}) that for any given $\delta\in \left(0, h\right)$ with $h=\min\{\lvert c_H\rvert,1\}$, one can choose $R$ sufficiently large such that 
	\begin{equation*}
	\left\lvert {(H_\lambda)}_2(x,y) - c_H y^d\right\rvert < \delta {\lvert y \rvert}^d
	\end{equation*}
	for all $(\lambda,x,y)\in V_R^+$, which in turn gives 
	\begin{equation}\label{est leq}
	\left( \lvert c_H\rvert - \delta \right) {\lvert y\rvert}^d < \left \lvert y_1^\lambda \right\rvert < \left( \lvert c_H\rvert + \delta \right) {\lvert y\rvert}^d
	\end{equation}
	in $V_R^+$.  Now (\ref{est leq}) gives that 
	\begin{equation}\label{y}
	\left\lvert y_1^\lambda \right\rvert>R 
	\end{equation}
	if $\left\lvert y \right\rvert>R$ with $R$  sufficiently large. Further, a similar calculation as in (\ref{est q}) gives that 
	\begin{equation}\label{x1}
	\left\lvert x_1^\lambda\right\rvert=\left\lvert {(H_\lambda)}_1(x,y)\right\rvert < \delta {\lvert y \rvert}^{d-1} 
	\end{equation}
	for $\lvert y \rvert$ sufficiently large. 
	From (\ref{est leq}) and (\ref{x1}), it follows that 
	\begin{equation}\label{xy}
	\lvert y_1^\lambda\rvert > \lvert x_1^\lambda\rvert
	\end{equation}
	for $(\lambda,x,y)\in V_R^+$ with $R$ sufficiently large. Further,  using (\ref{est leq}), we get that
	\begin{eqnarray*}
		\lvert y_1^\lambda\rvert &>& \left( \lvert c_H\rvert - \delta \right) {\lvert y\rvert}^d 
		> \left( \lvert c_H\rvert - \delta \right){\lvert y \rvert}^{d-1}{\lvert \lambda \rvert}^{\tilde{d}+1}
	\end{eqnarray*}
	for $(\lambda,x,y)\in V_R^+$. Thus if we take $\lvert y \rvert$ sufficiently large, then we get
	\begin{equation}\label{lambda}
	\left\lvert y_1^\lambda \right\rvert> {\left\lvert c\lambda\right\rvert}^{\tilde{d}+1}.
	\end{equation}
	Combining (\ref{y}), (\ref{xy}) and (\ref{lambda}), we get that
	\begin{equation}\label{subset}
	H(V_R^+)\subseteq V_R^+
	\end{equation}
	for $R$ sufficiently large.
	
	\medskip 
	\no 
	Let (\ref{estimate y}) holds for $1\leq n \leq k$, i.e., 
	\begin{equation}\label{induction}
	{\left(\lvert c_H\rvert-\delta\right)}^{\frac{d^n-1}{d-1}}{ \lvert y \rvert}^{d^n}<\left\lvert y_n^\lambda\right\rvert < {\left(\lvert c_H\rvert+\delta\right)}^{\frac{d^n-1}{d-1}}{ \lvert y \rvert}^{d^n}
	\end{equation}
	for  $1\leq n \leq k$.
	Using (\ref{subset}), it follows that  
	$$ 
	H^n(\lambda,x,y)=(c^n \lambda, x_n, y_n)=(c^n \lambda, {(H_\lambda^n)}_1(x,y), {(H_\lambda^n)}_2(x,y))\in V_R^+
	$$
	for $R>0$ sufficiently large and for all $n\geq 1$. Therefore, using (\ref{est leq}), we get 
	\begin{equation*}
	\left( \lvert c_H\rvert - \delta \right) {\lvert y_k^\lambda\rvert}^d < \left \lvert y_{k+1}^\lambda \right\rvert < \left( \lvert c_H\rvert + \delta \right) {\lvert y_k^\lambda\rvert}^d 
	\end{equation*}
	which, using induction hypothesis, gives 
	\begin{equation*}
	{\left(\lvert c_H\rvert-\delta\right)}^{\frac{d^{k+1}-1}{d-1}}{ \lvert y \rvert}^{d^{k+1}}<\left\lvert y_{k+1}^\lambda\right\rvert < {\left(\lvert c_H\rvert+\delta\right)}^{\frac{d^{k+1}-1}{d-1}}{ \lvert y \rvert}^{d^{k+1}}.
	\end{equation*}
	This finishes the proof.
\end{proof}

\begin{lem}\label{estK+}
	For  $(\lambda,x,y)\in K^+$, we have 
	\begin{equation*}
	\max \{x_n^\lambda, y_n^\lambda\}\leq {(LM)}^{(n+\frac{d}{d_m})}{\left[ \max \left\{R, \lvert x\rvert, {\lvert \lambda\rvert}^{\tilde{d}+1}\right\}\right]}^{{\left(1+\frac{d}{d_m}\right)}^{n}}.
	\end{equation*}
\end{lem}
\begin{proof}
	Let $(\lambda,x,y)\in K^+$ which implies that $H^n(\lambda,x,y)=(c^n \lambda, x_n^\lambda, y_n^\lambda)\notin V_R^+$ for all $n\geq 0$, i.e.,
	\begin{equation}\label{Mn}
	\left\lvert y_n^\lambda\right\rvert \leq M_n^\lambda =\max \left \{R, \left\lvert x_n^\lambda\right\rvert, {\left\lvert c^n\lambda\right\rvert}^{\tilde{d}+1} \right \}.
	\end{equation}
	Now
	\begin{eqnarray*}
		\left\lvert x_n^\lambda \right\rvert&=&\left\lvert {(H_{c^{n-1}\lambda})}_1(x_{n-1}^\lambda, y_{n-1}^\lambda)\right\rvert \\ 
		&=&\left\lvert \sum_{j+k=0}^{{d}/{d_m}} A_{jk}(c^{n-1}\lambda)({x_{n-1}^\lambda)} ^j {(y_{n-1}^\lambda)}^k\right\rvert \leq LM {\lvert c^{n-1} \lambda \rvert}_+^{\tilde{d}} {( M_{n-1}^\lambda)}^{{d}/{d_m}} \leq LM {(M_{n-1}^\lambda)}^{(1+\frac{d}{d_m})}
	\end{eqnarray*}
	for some $L,M>1$. The last inequality follows from (\ref{Mn}). Further, note that 
	\begin{equation*}
	{\left \lvert c^n \lambda\right \rvert}^{\tilde{d}+1}\leq {\lvert c \rvert}^{\tilde{d}+1}M_{n-1}^\lambda.
	\end{equation*}
	Also if 
	\[
	LM {(M_{n-1}^\lambda)}^{(1+\frac{d}{d_m})}< {\lvert c \rvert}^{\tilde{d}+1} M_{n-1}^\lambda,
	\]
	then it follows that 
	\[
	{(M_{n-1}^\lambda)}^{\frac{d}{d_m}} < \frac{{\lvert c \rvert}^{\tilde{d}+1}}{LM}, \text{i.e., } R^{\frac{d}{d_m}}<\frac{{\lvert c \rvert}^{\tilde{d}+1}}{LM}
	\]
	which leads to a contradiction if $R$ is sufficiently large.  Therefore
	\begin{eqnarray*}
		M_n^\lambda \leq LM {(M_{n-1}^\lambda)}^{(1+\frac{d}{d_m})}
	\end{eqnarray*}
	and by induction we get 
	\begin{equation*}
	M_n^\lambda \leq {(LM)}^{(n+\frac{d}{d_m})}{\left[ \max \left\{R, \lvert x\rvert, {\lvert \lambda\rvert}^{\tilde{d}+1}\right\}\right]}^{{\left(1+\frac{d}{d_m}\right)}^{n}}.
	\end{equation*}
\end{proof}
\begin{lem} \label{global est}
	There exists a constant $K>1$ depending on the coefficients of $H$ such that for all $n\geq 1$ and for all $(\lambda,x,y)\in \mathbb{C}^3$, we have 
	\begin{equation*}
	\max \left\{\left\lvert x_n^\lambda\right\rvert, \left\lvert y_n^\lambda\right\rvert\right \} \leq {\left(K{\lvert\lambda\rvert}_+^{\tilde{d} } {\lvert c \rvert}_+^{\tilde{d} } {\lVert (x,y)\rVert}_+\right)}^{d^n}
	\end{equation*}
	where ${\lvert x \rvert}_+ \big/ {\lVert (x,y)\rVert}_+=\max \{{\lvert x \rvert}\big/ {\lVert (x,y)\rVert},1\}$, for $x, y\in \mathbb{C}$.
\end{lem}
\begin{proof}
	Using (\ref{ABjk}), it follows that 
	\[
	\max \left\{\left\lvert x_1^\lambda\right\rvert,\left\lvert y_1^\lambda\right\rvert\right \} \leq K {\lvert \lambda \rvert}_+^{\tilde{d}} {\lVert (x,y)\rVert}_+^{d}
	\]
	for some $K>1$. Further, inductively we get that 
	\[
	\max \left\{\left\lvert x_n^\lambda\right\rvert, \left\lvert y_n^\lambda\right\rvert\right \} \leq 
	K^{(d^n-1)} {\lvert \lambda \rvert}_+^{\tilde{d}(d^n-1)} {\left({\lvert c \rvert}_+^{\tilde{d}}\right)}^{A_n}{\lVert (x,y)\rVert}_+^{d^n}
	\]
	where 
	$$
	A_n=\sum_{j=1}^{n-1} d^{j-1} \leq d^{n-2}\sum_{j=1}^{n-1} \frac{n-j}{d^{n-j-1}} \leq d^{n-2} \sum_{j=1}^\infty \frac{j}{d^{j-1}} \leq d^n.
	$$
	This finishes the proof.
\end{proof}
\begin{thm}\label{Green_Skew}
	The sequence of functions $G_{n,H}^\pm$ converges to $G_H^\pm$ uniformly on compact sets in $\mathbb{C}^3$. The function $G_H^\pm$ is plurisubharmonic in $\mathbb{C}^3$ and pluriharmonic in $\mathbb{C}^3\setminus K_H^\pm$ satisfying 
	\begin{equation*}
	G_H^\pm \circ H^{\pm 1}=d G_H^\pm
	\end{equation*}
	in $\mathbb{C}^3$. Further, $G_H^\pm>0$ in $U_H^\pm$ and it vanishes precisely in $K_H^\pm$.
\end{thm}
\begin{proof}
	It follows from Lemma \ref{estK+} that $G_{n,H}^+ \rightarrow 0$ as $n\rightarrow \infty$ in $K_H^+$. Further, it follows from Lemma \ref{est y} that 
	\[
	G_{n,H}^+(\lambda,x,y)=\frac{1}{\tilde{d} d^{n-1}} \log^+ \lVert y_n^\lambda\rVert
	\]
	in $U_H^+$. Note that (\ref{Green def}) gives 
	\[
	\left\lvert G_{{(n+1)},H}^+(\lambda,x,y)-G_{n,H}^+(\lambda,x,y)\right\rvert\leq \frac{1}{\tilde{d}d^n}\log^+ \frac{\left\lvert y_{n+1}^\lambda\right\rvert}{{\left\lvert y_{n}^\lambda\right\rvert}^d}\leq \frac{C}{\tilde{d}d^n} 
	\]
	for some $C>1$, which in turn gives that $G_{n,H}^+$ converges uniformly to the function $G_H^+$ in  $U_H^+$. Clearly, $G_H^+$ is pluriharmonic in $U_H^+$ and it follows from Lemma \ref{est y} that $G_H^+>0$ in $U_H^+$.  Note that $U_H^+$ is completely invariant under $H$ which  shows that the set  $K_H^+$ is also completely invariant under $H$. It follows from (\ref{Green def}) that 
	\begin{equation}\label{Green funct}
	G_H^+\circ H=d G_H^+
	\end{equation}
	in $\mathbb{C}^3$. Further, it follows from Lemma \ref{est y} that
	\begin{equation} \label{est Gn}
	G_{n,H}^+ (\lambda,x,y)< M+ d\log^+ \lvert \lambda\rvert + \log^+ \lVert (x,y)\rVert
	\end{equation}
	in $\mathbb{C}^3$ for each $n\geq 1$ where $M$ is some constant depending on $H$.  Now let $G_{H*}^+$ is the upper semicontinuous regularization of $G_H^+$. So, $G_{H*}^+$ is plurisubharmonic in $\mathbb{C}^3$. Further, since $G_H^+$ is pluriharmonic in $U_H^+ \cup {\rm{int}}(K_H^+)$, it follows that $G_{H*}^+=G_H^+$ in $U_H^+ \cup {\rm{int}}(K_H^+)$. In addition, $G_{H*}^+$ enjoys the same functorial property as $G_H^+$ indicated in (\ref{Green funct}), i.e., 
	\begin{equation} \label{Green *funct}
	G_{H*}^+\circ H=d G_{H*}^+
	\end{equation}
	in $\mathbb{C}^3$ and using (\ref{est Gn}), we get
	\[
	G_{H*}^+\circ H^n(\lambda,x,y)= G_{H*}^+(c^n \lambda, x_n^\lambda, y_n^\lambda)\leq 
	M+ d\log^+ \lvert c^n \lambda\rvert + \log^+ \lVert (x_n^\lambda,y_n^\lambda)\rVert
	\]
	for $(\lambda,x,y)\in K_H^+$.
	Now combining Lemma \ref{global est} and (\ref{Green *funct}), we get that
	\[
	\frac{1}{d^n}{G_{H*}^+ \circ H^n(\lambda,x,y)}\rightarrow 0
	\]
	as $n\rightarrow \infty$ for all $(\lambda,x,y)\in K_H^+$. Therefore, $G_{H*}^+(\lambda,x,y)=0$ for $(\lambda,x,y)\in K_H^+$. Thus $G_{H*}^+=G_H^+$ in $\mathbb{C}^3$ which proves that $G_H^+$ is  plurisubharmonic in $\mathbb{C}^3$. Now using Hartog's Lemma we conclude that $G_{n,H}^+$ converges uniformly to $G_H^+$ on compact sets in $\mathbb{C}^3$. Running the similar set of arguments, we can prove the analogous result for $G_H^-$.
\end{proof}
\no 
Recall that
$$
V_R^{-}=\{(\lambda,x,y)\in \mathbb{C}^3: \lvert x\rvert > \max \{R, \lvert y \rvert\}, \lvert \lambda \rvert <1\}.
$$
Running the similar set of arguments as in Lemma \ref{est y}, we get the following.
\begin{lem}\label{est x}
	For any $\delta\in \left(0, h\right)$ where $h=\min\{\lvert c_H'\rvert, 1\}$, there exists $R_0=R_0(\delta)>1$ such that
	\begin{equation*}
	H^{-1}(V_R^-)\subseteq V_R^-
	\end{equation*}
	for all $R\geq R_0$. Further,
	\begin{equation}\label{estimate x}
	{\left(\lvert c_H'\rvert-\delta \right)}^{\frac{d^n-1}{d-1}}{\lvert x \rvert}^{d^n}<\lvert \tilde{x}_n^\lambda \rvert< {\left(\lvert c_H'\rvert+\delta\right)}^{\frac{d^n-1}{d-1}}{\lvert x \rvert}^{d^n}
	\end{equation}
	for all $(\lambda,x,y)\in V_R^-$.
\end{lem}

\begin{lem} \label{estK-}
For each $(\lambda,x,y)\in K_H^-$, $G_{n,H}^-(\lambda,x,y)\rightarrow 0$ as $n\rightarrow \infty$.
\end{lem}
\begin{proof}
Let $(\lambda,x,y)\notin U_H^-$ which implies $H^{-n}(\lambda,x,y)=(c^{-n} \lambda, \tilde{x}_n^\lambda,\tilde{y}_n^\lambda)\notin V_R^-$ for any $n\in \mathbb{N}$. For a given $\lambda\in \mathbb{C}$, choose $n_0$ large enough such that $\lvert c^{-n_0}\lambda\rvert <1$. Now let 
\begin{equation}\label{Mnlambda}
\lvert \tilde{x}_n^\lambda \rvert \leq M_n^\lambda = \max \{R, \lvert \tilde{y}_n^\lambda \rvert \}
\end{equation}
for $n\geq n_0$.  Now using (\ref{abjk}), it follows that
\begin{equation}\label{tildeYn}
\tilde{y}_n^\lambda={\left(H_{c^{-n}\lambda}^{-1}\right)}_2(\tilde{x}_{n-1}^\lambda,\tilde{y}_{n-1}^\lambda)=\sum_{i+j=0}^{{d}/{d_1}} B_{jk}'(c^{-n}\lambda) \tilde{x}_{n-1}^{\lambda^j}\tilde{y}_{n-1}^{\lambda^k}.
\end{equation}
Note that $B_{jk}$'s are polynomials in $\mathbb{C}$ with degree at most $\tilde{d}$. Therefore, there exists $L>1$ such that 
\begin{equation}\label{estimate pol2}
\left\lvert B_{jk}(\lambda)\right\rvert < L{\lvert \lambda\rvert}_+^{\tilde{d}}
\end{equation}
for all $j,k$'s with $0\leq (j+k) \leq (d-1)$ where ${\lvert\lambda\rvert}_+=\max\{\lvert \lambda\rvert,1\}$. Now using (\ref{Mnlambda}), (\ref{tildeYn}) and (\ref{estimate pol2}), we get that
\begin{eqnarray*}
\left\lvert \tilde{y}_n^\lambda\right\rvert
=\left\lvert {\left(H_{c^{-n}\lambda}^{-1}\right)}_2(\tilde{x}_{n-1}^\lambda,\tilde{y}_{n-1}^\lambda)\right\rvert\leq 
\left\lvert \sum_{i+j=0}^{{d}/{d_1}} B_{jk}'(c^{-n}\lambda) \tilde{x}_{n-1}^{\lambda^j}\tilde{y}_{n-1}^{\lambda^k}\right\rvert
 \leq {LM_{n-1}^\lambda}^{{d}/{d_1}}
\end{eqnarray*}
for some $L>1$ and for all $n\geq n_0$.  Thus by induction, we get 
\begin{equation} \label{estMnlamda}
M_n^\lambda \leq L^{\frac{\tilde{d}^{(n-n_0)}-1}{\tilde{d}-1}} {\left(M_{n_0}^\lambda\right)}^{\hat{d}^{n-n_0}} 
\end{equation}
where $\hat{d}={d}/{d_1}$.  Using the inequality (\ref{estMnlamda}), we get that $G_{n,H}^{-}(\lambda,x,y)\rightarrow 0$ as $n\rightarrow \infty$.
\end{proof}
\no 
Since $G_H^\pm$ are pluriharmonic in $\mathbb{C}^3 \setminus K_H^\pm$ and in $\mathbb{C}^3 \setminus J_H^\pm$, the following proposition holds. Note that the similar result holds for a single H\'{e}non map (Proposition 3.8, \cite{BS1}). Since, one can prove the following proposition using the same arguments as in the case of a single H\'{e}non map, we are omitting the proof.
\begin{prop}\label{pluricomp}
The functions $G_H^\pm$ are the pluricomplex Green functions of the sets $K_H^\pm$ and of the sets $J_H^\pm$.
\end{prop}

\begin{prop}\label{P3}
	For $(\lambda,x,y)\in V_R^+$, 
	\begin{equation*}
	\frac{x_n^\lambda}{y_n^\lambda} \rightarrow 0
	\end{equation*}	
	as $n\rightarrow \infty$ and  for $(\lambda,x,y)\in V_R^-$, 
	\begin{equation*}
	\frac{y_n^\lambda}{x_n^\lambda} \rightarrow 0
	\end{equation*}	
	as $n\rightarrow \infty$.
\end{prop}	
\begin{proof}
	Since $H(V_R^+)\subseteq V_R^+$,  using (\ref{ABjk}) we get that
	\begin{equation}\label{xnn}
	\left\lvert x_n^\lambda \right\rvert= \left\lvert \sum_{j+k=0}^{{d}/{d_m}} A_{jk}(c^{n-1}\lambda){x_{n-1}^{\lambda ^j}} {y_{n-1}^{\lambda^k}}\right\rvert \leq LM {\lvert c^{n-1} \lambda \rvert}_+^{\tilde{d}}  {\lvert y_{n-1}^\lambda \rvert}^{d-1} 
	\end{equation}
	for some $L>1$. Also from (\ref{est leq}), it follows that for $\delta>0$ small enough,
	\begin{equation}\label{ynn}
	\lvert y_n^\lambda \rvert > \left( \lvert c_H\rvert -\delta\right){ \lvert y_{n-1}^\lambda \rvert}^{d}.
	\end{equation}
	Combining (\ref{xnn}) and (\ref{ynn}), it follows that 
	\begin{equation}\label{R}
	\frac{\left \lvert y_n^\lambda \right\rvert}{\left \lvert x_n^\lambda\right\rvert} > \frac {K \lvert y_{n-1}^\lambda \rvert}{ {\left\lvert c^{n-1} \lambda \right\rvert}_+^{\tilde{d}} } 
	\geq \frac{K}{{\left\lvert c^{n-1} \lambda \right\rvert}_+^{\tilde{d}} } {(\lvert c_H \rvert-\delta)}^{\frac{d^{n-1}-1}{d-1}} {\lvert y \rvert}^{d^{n-1}}
	\end{equation}
	for some $K>1$. Now since $\lvert y \rvert>R>1$,  it follows from (\ref{R}) that
	\begin{equation*}
	\frac{x_n^\lambda}{y_n^\lambda} \rightarrow 0
	\end{equation*}	
	as $n\rightarrow \infty$. Using the similar arguments, one can show that for $(\lambda,x,y)\in V_R^-$
	\begin{equation*}
	\frac{\tilde{y}_n^\lambda}{\tilde{x}_n^\lambda} \rightarrow 0
	\end{equation*}	
	as $n\rightarrow \infty$.
\end{proof}
\begin{rem}\label{P3rem}
Therefore it follows from Proposition \ref{P3} that if $(\lambda,x,y)\in V_R^+$, then  $H^n(\lambda,x,y)\rightarrow [0:0:1:0]$ in $\mathbb{P}^3$ as $n\rightarrow \infty$. Similarly, if $(\lambda,x,y)\in V_R^-$, then  $H^{-n}(\lambda,x,y)\rightarrow [0:1:0:0]$ in $\mathbb{P}^3$ as $n\rightarrow \infty$.
\end{rem}

\medskip 
\no 


\begin{center}
	\it{Case 2: }$\lvert c \rvert < 1$  and  \it{Case 3: }$\lvert c \rvert =1$
\end{center}
The analogues of the above results can be proved for the skew products of H\'{e}non maps with $\lvert c \rvert \leq 1$ and to do so one needs to consider the cases $\lvert c\rvert<1$ and $\lvert c\rvert=1$ separately. In particular, Remark \ref{P3rem} holds for all $0 \neq c\in \mathbb{C}$. Recall that we need to modify the  filtrations in each cases. 

\section{Rigidity of skew products of H\'{e}non maps with non-comapct parameter space}
\subsection*{B\"{o}ttcher coordinates and its relation with Green functions}

\begin{prop}\label{Bot}
	Let $H:\mathbb{C}^3\rightarrow \mathbb{C}^3$ be a skew products of H\'{e}non maps. Then there exist non-vanishing holomorphic functions 
	$\phi_H^\pm: V_R^\pm \rightarrow \mathbb{C}$  such that 
	\[
	\phi_H^+\circ H(\lambda,x,y)={(c_H)}^{\frac{d}{\tilde{d}}}{(\phi_H(\lambda,x,y))}^d
	\]
	in $V_R^+$ and 
	\[
	\phi_H^-\circ H^{-1}(\lambda,x,y)={(c_H')}^{\frac{d}{\tilde{d}}}{(\phi_H^-(\lambda,x,y))}^d
	\]
	in $V_R^-$. Further,
	\[
	\phi_H^+(\lambda,x,y)\sim y \text{ as } \lVert(\lambda,x,y)\rVert\rightarrow \infty \text{ in } V_R^+
	\]
	and
	\[
	\phi_H^-(\lambda,x,y)\sim x \text{ as } \lVert(\lambda,x,y)\rVert\rightarrow \infty \text{ in } V_R^-.
	\]
	
\end{prop}

\begin{proof}
	Consider the following telescoping product 
	\begin{equation}\label{product}
	y. \frac{{(y_1^\lambda)}^{\frac{1}{\tilde{d}}}}{y}. \cdots. \frac{{(y_{n+1}^\lambda)}^{\frac{1}{\tilde{d}d^{n}}}}{{(y_{n}^\lambda)}^{\frac{1}{\tilde{d}d^{n-1}}}}. \cdots.
	\end{equation}
	Now by (\ref{ABjk}), 
	\begin{equation*}
	y_{n+1}^\lambda= {(H_\lambda^{n+1})}_2(x,y)= c_H {(y_n^\lambda)}^d+ q_{c^n \lambda}(x_n^\lambda, y_n^\lambda).
	\end{equation*}
	Hence, 
	\begin{eqnarray*}
	\frac{{(y_{n+1}^\lambda)}^{\frac{1}{\tilde{d}d^{n}}}}{{(y_{n}^\lambda)}^{\frac{1}{\tilde{d}d^{n-1}}}}&=&\frac{{\left(c_H {(y_n^\lambda)}^d+ q_{c^n \lambda}(x_n^\lambda, y_n^\lambda)\right)}^{\frac{1}{\tilde{d} d^n}}} {{\left(y_{n}^\lambda\right)}^{\frac{1}{\tilde{d}d^{n-1}}}}\nonumber \\
	&=&{ \left( c_H+ \frac{q_{c^n \lambda}(x_n^\lambda, y_n^\lambda)} {{(y_n^\lambda )}^d} \right)}^{\frac{1}{\tilde{d} d^n}} \nonumber \\
	&=& {c_H}^{\frac{1}{\tilde{d} d^n}} { \left( 1+ \frac{q_{c^n \lambda}(x_n^\lambda, y_n^\lambda)} {c_H{(y_n^\lambda )}^d} \right)}^{\frac{1}{\tilde{d} d^n}}.
	\end{eqnarray*}
	Now
	\begin{equation}\label{Qest}
	\left \lvert\frac{ q_{c^n \lambda}(x_n^\lambda, y_n^\lambda)}{c_H {(y_n^\lambda)}^d}\right \rvert  \leq  \frac{1}{\left \lvert c_H {(y_n^\lambda)}^d \right \rvert}\sum_{j+k=0}^{d-1} \left \lvert B_{jk}(c^n \lambda) x_n^{\lambda^j}  y_n^{\lambda^k}\right \rvert
	\leq  \frac{L{\lvert c^n \lambda \rvert}_+^{\tilde{d}}}{\lvert c_H y_n^\lambda\rvert } <\frac{1}{2}
	\end{equation}
	for $\lvert y \rvert> R>>1$ and for all $n\geq 1$. The second last inequality follows from (\ref{ABjk}) and from the fact that $H(V_R^+)\subseteq V_R^+$. The last inequality follows from (\ref{estimate y}). 
	
	\medskip 
	\no 
	Let  ${ \left( 1+ \frac{q_{c^n \lambda}(x_n^\lambda, y_n^\lambda)} {c_H{(y_n^\lambda )}^d} \right)}^{\frac{1}{\tilde{d} d^n}}$ be the principal branch of $\tilde{d}d^n$-th root of  ${ \left( 1+ \frac{q_{c^n \lambda}(x_n^\lambda, y_n^\lambda)} {c_H{(y_n^\lambda )}^d} \right)}$. 
	Now note that the convergence of the product in (\ref{product}) is equivalent to the convergence of the series 
	\begin{equation}\label{Log}
	{\rm{Log}} \ y+ {\rm{Log}} \left(\frac{y_1^\lambda}{y}\right)+\cdots + {\rm{Log}} \left(\frac{{(y_{n+1}^\lambda)}^{\frac{1}{\tilde{d}d^{n}}}}{{(y_{n}^\lambda)}^{\frac{1}{\tilde{d}d^{n-1}}}}\right)+\cdots.
	\end{equation}
	Since, (\ref{Qest}) holds, the above series in (\ref{Log}) is absolutely convergent which shows that the product in (\ref{product}) is convergent. 
	
	Since the infinite product in (\ref{product}) is convergent,  we can define the function $\phi_H^+: V_R^+ \rightarrow \mathbb{C}$ as follows:
	\begin{equation*}
	\phi_H^+(\lambda, x,y)= {(c_H)}^{-\frac{d}{\tilde{d}(d-1)}}\lim_{n\rightarrow \infty} {(y_n^\lambda)}^{\frac{1}{\tilde{d}d^{n-1}}}
	={(c_H)}^{-\frac{d}{\tilde{d}(d-1)}}\left(y. \frac{{(y_1^\lambda)}^{\frac{1}{\tilde{d}}}}{y}. \cdots. \frac{{(y_{n+1}^\lambda)}^{\frac{1}{\tilde{d}d^{n}}}}{{(y_{n}^\lambda)}^{\frac{1}{\tilde{d}d^{n-1}}}}. \cdots\right).
	\end{equation*} 
	Since
	\[
	\left \lvert\frac{ q_{c^n \lambda}(x_n^\lambda, y_n^\lambda)}{c_H {(y_n^\lambda)}^d}\right \rvert  \leq  
	\frac{L{\lvert c^n \lambda \rvert}_+^{\tilde{d}}}{\lvert c_H y_n^\lambda\rvert } \leq \frac{L{\lvert c^n \lambda \rvert}_+^{\tilde{d}}}{\lvert c_H \rvert {(\lvert c_H\rvert-\delta)}^{\frac{d^n-1}{d-1}}{\lvert y \rvert}^{d^n}}, 
	\]
	which we get combining (\ref{estimate y}) and (\ref{est q}), it follows that 
	\begin{equation*}
	\phi_H^+(\lambda, x,y) \sim y
	\end{equation*}
	as $\lVert (\lambda,x,y)\rVert \rightarrow \infty$ in $V_R^+$.
	
	\medskip 
	\no 
	Note that 
	\begin{eqnarray*}
		H^{n+1}(\lambda,x,y)&=&H(c^n \lambda, x_n^\lambda, y_n^\lambda)\\
		&=& (c^{n+1} \lambda, x_{n+1}^\lambda, c_H {(y_n^\lambda)}^d+q_{c^n \lambda} (x_n^\lambda, y_n^\lambda)).
	\end{eqnarray*}
	Hence, 
	\[
	{(H_{c\lambda}^n)}_2(H(\lambda,x,y))=c_H{(H_\lambda^n)}_2^d(\lambda,x,y)(1+L(\lambda,x,y))
	\]
	where $L(\lambda,x,y)\rightarrow 0$ as $\lVert(\lambda,x,y)\rVert \rightarrow \infty$ in $V_R^+$. Thus 
	\begin{eqnarray*}
	\phi_H^+(H(\lambda,x,y))&=&c_H^{-\frac{d}{\tilde{d}(d-1)}}\lim_{n\rightarrow \infty} {(H_{c\lambda}^n)}_2^{\frac{1}{\tilde{d}d^{n-1}}}(H(\lambda,x,y)) \nonumber\\
	&=& c_H^{-\frac{d}{\tilde{d}(d-1)}}\lim_{n\rightarrow \infty}{ \left(c_H {(H_\lambda^n)}_2^d(\lambda,x,y)(1+L(\lambda,x,y))\right)}^{\frac{1}{\tilde{d} d^{n-1}}}\nonumber \\
	&=& c_H^{\frac{d}{\tilde{d}}} \lim_{n\rightarrow \infty} c_H^{\frac{1}{\tilde{d} d^{n-1}}}{ \left(c_H^{-\frac{d}{\tilde{d}(d-1)}} {(H_\lambda^n)}_2^{\frac{1}{\tilde{d}d^{n-1}}}(\lambda,x,y){(1+L(\lambda,x,y))}^{\frac{1}{\tilde{d}d^n}}\right)}^d \nonumber \\
	&=& c_H^{\frac{d}{\tilde{d}}} {(\phi_H^+(\lambda,x,y))}^d.
	\end{eqnarray*}
	
	Now $H^{-n}(\lambda,x,y)=(c^{-n}\lambda, \tilde{x}_n^\lambda, \tilde{y}_n^\lambda)$ where $ \tilde{x}_n^\lambda$ is a polynomial in $x$ and $y$ of degree $\tilde{d}d^{n-1}$ which is precisely the  degree of $H^{-n}$.
	Consider the following telescoping product 
	\begin{equation}\label{product x}
	x. \frac{{(\tilde{x}_1^\lambda)}^{\frac{1}{\tilde{d}}}}{x}. \cdots. \frac{{(\tilde{x}_{n+1}^\lambda)}^{\frac{1}{\tilde{d}d^{n}}}}{{(\tilde{x}_{n}^\lambda)}^{\frac{1}{\tilde{d}d^{n-1}}}}. \cdots
	\end{equation}
	which can be shown to be convergent as before. We define 
	\[
	\phi_H^-(\lambda, x,y)= {(c_H')}^{-\frac{d}{\tilde{d}(d-1)}}\lim_{n\rightarrow \infty} {(\tilde{x}_n^\lambda)}^{\frac{1}{\tilde{d}d^{n-1}}}.
	\]
	That 
	$$
	\phi_H^-(\lambda,x,y)\sim x
	$$ 
	as $\lVert (\lambda,x,y)\rVert\rightarrow \infty$ in $V_R^-$ and
	$$
	\phi_H^-(H^{-1}(\lambda,x,y))={(c_H')}^{\frac{d}{\tilde{d}}} {(\phi_H^-(\lambda,x,y))}^d
	$$ 
	in $V_R^-$, can be shown in the similar fashion as in the case of $\phi_H^+$. 
\end{proof}
Further, it follows from (\ref{Green def}) that
\begin{equation} \label{G+phi}
G_H^+(\lambda,x,y)=\log \lvert \phi_H^+\rvert+ \frac{d}{\tilde{d}(d-1)} \log \lvert c_H\rvert
\end{equation}
in $V_R^+$ and 
\begin{equation}\label{G-phi}
G_H^-(\lambda,x,y)=\log \lvert \phi_H^-\rvert+ \frac{d}{\tilde{d}(d-1)} \log \lvert c_H'\rvert
\end{equation}
in $V_R^-$.

\subsection*{Proof of the Theorem \ref{skew-ncompact}:} Since $J_H^\pm=J_F^\pm$, it follows from Proposition \ref{pluricomp} that the pluricomplex Green functions of these sets coincide, i.e., 
\begin{equation*}
G_H^\pm= G_F^\pm.
\end{equation*}
At this point, we refer the readers to the proof of Theorem 1.1 in \cite{rigidity} (or Theorem \ref{rigidity of shifts} in the present paper) which can be adapted to show that 
\begin{equation} \label{relationFH}
{(F\circ H)}_1= \delta^+ {(H\circ F)}_1 \text{ and } {(F\circ H)}_2= \delta^- {(H\circ F)}_2 
\end{equation}
in $\mathbb{C}^3$ for some $\delta_+, \ \delta_- \in \mathbb{C}$ with $\lvert \delta^\pm \rvert=1$. Note that Proposition \ref{bott} along with the relations (\ref{G+phi}) and (\ref{G-phi}) play a crucial role to establish (\ref{relationFH}).
This finishes the proof.

\section{Rigidity of skew products of H\'{e}non maps with compact parameter space}
\no 
In this section, we consider skew products of H\'{e}non maps fibered over a compact metric space $M$, i.e., the maps of the form
\begin{equation}\label{compHenon}
H(\lambda,x,y)=(\sigma_H (\lambda), H_\lambda(x,y))
\end{equation}
for $(\lambda,x,y)\in M\times \mathbb{C}^2$ where $\sigma_H$ acts as an homeomorphism on $M$. Recently, the  dynamics of these  maps  has been studied in \cite{PV} and \cite{PV1}. For each $\la \in M$ and for each $n\geq 1$, let 
\[
{H}_\lambda^n= H_{\si_H^{n-1}(\la)} \circ H_{\si_H^{n-2}(\la)} \circ \cdots \circ H_{\la}
\]
and 
\[
{H}_\lambda^{-n}=H^{-1}_{\si_H^{-n}(\la)} \circ H^{-1}_{\si_H^{-(n-1)}(\la)} \circ \cdots \circ H^{-1}_{\si_H^{-1}(\la)}.
\]
We define the non-escaping sets and escaping sets  as follow:
\[
I^{\pm}_{H,\lambda} = \left\{ (x, y) \in \mbb C^2 : \Vert {H}_{\la}^{\pm n}(x,y) \Vert \ra \infty \; \text{as} \; n \ra \infty \right\}
\]
and
\[
K^{\pm}_{H,\la} = \left\{ (x, y) \in \mbb C^2 : \text{the orbit} \; \big( {H}_{\la}^{\pm n}(x,y)\big)_{n \ge 0} \; \text{is bounded in} \; \mbb C^2  \right\}
\]
respectively. Further, for each $\la \in M$, we define $J^{\pm}_{H, \la} = \pa K^{\pm}_{H, \la}$. Further, define 
\begin{equation}\label{fiberedGreen}
G_{n,H,\lambda}^\pm (x,y)=\frac{1}{d^n} \log^+ \left\lVert H_\lambda^{\pm n}(x,y) \right\rVert
\end{equation}
for each $n\geq 1$ and for all $(x,y)\in \mathbb{C}^2$.
Recall form \cite{PV1} that the sequence of functions defined in (\ref{fiberedGreen}) converge uniformly in compact to the plurisubharmonic functions $G_{H,\lambda}^\pm$ in $\mathbb{C}^2$.  Further, note that (see \cite{PV1}) a uniform filtration defined as follow:
\begin{align*}
V^+_R &= \{ (x,y) \in \mathbb C^2: \vert x \vert < \vert y \vert, \vert y \vert > R \},\\
V^-_R &= \{ (x,y) \in \mathbb C^2: \vert y \vert < \vert x \vert, \vert x \vert > R \},\\
V_R &= \{ (x, y) \in \mathbb C^2: \vert x \vert, \vert y \vert \le R \}.
\end{align*}
with $R>0$ sufficiently large  works uniformly for all $H_\lambda$'s.
\subsection{Fibered B\"{o}ttcher coordinates}
\no 
For each $\lambda\in M$, let 
\[
H_\lambda(x,y)=({(H_\lambda)}_1(x,y), {(H_\lambda)}_2(x,y))
\]
be a composition of generalized H\'{e}non maps as described before. 
Clearly, the degree of ${(H_\lambda)}_1$ is strictly less than that of ${(H_\lambda)}_2$ when considered as  polynomials in $y$. Further, let, for each $\lambda\in M$
\begin{equation}\label{y coor}
{(H_\lambda)}_2(x,y)=c_H y^d + q_\lambda (x,y)
\end{equation}
where $c_H=\Pi_{j=1}^m c_j^{d_{(j+1)}\cdots  d_m}$ with $d_{j+1}\cdots d_m=1$ when $j=m$, $d=d_n \cdots d_1$ and $q_\lambda$ a polynomial in $x,y $ of degree strictly less than $d$.

\medskip 
\no 
Similarly, let 
\[
{(H_\lambda)}^{-1}(x,y)=({(H_\lambda)}_1'(x,y), {(H_\lambda)}_2'(x,y))
\]
where degree of $(H_\lambda^2)'$ is strictly less than that of $(H_\lambda^1)'$ considering  as  polynomials in $x$.  As before, one can write
\begin{equation}\label{x coor}
{(H_\lambda)}_1'(x,y)= c_H' x^d + q_\lambda'(x,y)
\end{equation}
where $c_H'=\Pi_{j=1}^m (c_j \delta_j^{-1})^{d_{(j-1)}\cdots d_1}$ where $d_{j-1}\cdots d_1=1$ when $j=1$ and $q_\lambda'$ a polynomial in $x, y$ of degree strictly less than $d$.

\begin{prop}\label{fiberedBott}
	Let $H$ be a skew products of H\'{e}non maps of the form (\ref{compHenon}), then for each $\lambda\in M$, there exist  non-vanishing analytic functions $\phi_{H,\lambda}^\pm: V_R^\pm \rightarrow \mathbb{C}$  such that 
	\[
	\phi_{H,\sigma_H(\lambda)}^+(H_\lambda(x,y))=c_H {(\phi_{H,\lambda}^+(x,y))}^d
	\]
	in $V_R^+$ and 
	\[
	\phi_{H,\lambda}^-\circ H_\lambda^{-1}={c_H'} {\left(\phi_{H,\sigma_H(\lambda)}^-(x,y)\right)}^d
	\]
	in $V_R^-$. Further,
	\[
	\phi_{H,\lambda}^+(x,y)\sim y \text{ as } \lVert(x,y)\rVert\rightarrow \infty \text{ in } V_R^+
	\]
	and
	\[
	\phi_{H,\lambda}^-(x,y)\sim x \text{ as } \lVert(x,y)\rVert\rightarrow \infty \text{ in } V_R^-.
	\]
	
\end{prop}

\begin{proof}
	 Let,
	$$
	H_\lambda^n(x,y)=\left({(H_\lambda^n)}_1(x,y),{(H_\lambda^n)}_2(x,y)\right)
	$$
	for $(x,y)\in \mathbb{C}^2$ 
	and
	\begin{equation*}
	y_{n,\lambda}={(H_\lambda^n)}_2  \text{ and } x_{n,\lambda}={(H_\lambda^n)}_1
	\end{equation*}
	are polynomials in $x$ and $y$ in $\mathbb{C}^2$. Further, note that the degree of $y_{n,\lambda}$ is $d^n$. 
	
	\medskip 
	\no 
	Now consider the following telescoping series
	\begin{equation}\label{tele}
	y.\frac{y_{1,\lambda}^{\frac{1}{d}}}{y}. \cdots .\frac{y_{n+1,\lambda}^{\frac{1}{d^{n+1}}}}{y_{n,\lambda}^{\frac{1}{d^{n}}}}.\cdots.
	\end{equation}
	Note that form (\ref{y coor}), we get the followings:
	
	\begin{eqnarray}\label{bott}
	\frac{y_{n+1,\lambda}^{\frac{1}{d^{n+1}}}}{y_{n,\lambda}^{\frac{1}{d^{n}}}}&=&\frac{{\left(c_H y_{n, \lambda}^d+q_{\sigma_H^n(\lambda)}(x_{n,\lambda},y_{n,\lambda})\right)}^{\frac{1}{d^{n+1}}}}{y_{n,\lambda}^{\frac{1}{d^n}}}\nonumber\\
	&=&{\left(c_H+\frac{q_{\sigma_H^n(\lambda)}(x_{n,\lambda},y_{n,\lambda})}{y_{n,\lambda}^d}\right)}^{\frac{1}{d^{n+1}}}\nonumber\\
	&=& {c_H}^{\frac{1}{d^{n+1}}}{\left(1+\frac{q_{\sigma_H^n(\lambda)}(x_{n,\lambda},y_{n,\lambda})}{c_H y_{n,\lambda}^d}\right)}^{\frac{1}{d^{n+1}}}.
	\end{eqnarray}
	\no 
	Since the parameter space $M$ is compact and  and the coefficients of $q_\lambda$'s vary continuously in $\lambda$, we can choose an $R>0$ sufficiently large such that    
	\[
	\left\lvert \frac{q_{\sigma_H^n(\lambda)}(x_{n,\lambda},y_{n,\lambda})}{c_H y_{n,\lambda}^d}\right\rvert <1
	\]
	for all $\lvert y\rvert >R$ and for all $n\geq 1$.
	Hence, ${\left(1+\frac{q_{\sigma_H^n(\lambda)}(x_{n,\lambda},y_{n,\lambda})}{c_H y_{n,\lambda}^d}\right)}^{\frac{1}{d^{n+1}}}$, the principal branch of the $d^{n+1}$-th root of 
	${\left(1+\frac{q_{\sigma_H^n(\lambda)}(x_{n,\lambda},y_{n,\lambda})}{c_H y_{n,\lambda}^d}\right)}$
	is well-defined for all $\lambda\in M$ and for all $n\geq 1$.
	
	\medskip 
	\no 
	Note that, for a fixed $\lambda\in M$, there exist $L,\  \tilde{L}>0$ 
	such that
	\begin{equation}\label{zero}
	\left\lvert\frac{q_{\sigma_H^n(\lambda)}(x_{n,\lambda},y_{n,\lambda})}{y_{n,\lambda}^d}\right\rvert\leq \frac{L}{\lvert y_{n,\lambda} \rvert}\leq \frac{\tilde{L}}{{\lvert y\rvert}^{d^n}}
	\end{equation}
	for $R>0$  large enough. Now since, 
	convergence of the series
	\begin{equation}\label{tele1}
	y.\frac{y_{1,\lambda}^{\frac{1}{d}}}{y}. \cdots .\frac{y_{n+1,\lambda}^{\frac{1}{d^{n+1}}}}{y_{n,\lambda}^{\frac{1}{d^{n}}}}.\cdots
	\end{equation}
	implies the convergence of the series
	\begin{equation}\label{Log1}
	{\rm{Log}}\ y+\frac{1}{d} {\rm{Log}} \left(\frac{{y_{1,\lambda}}^{\frac{1}{d}}}{y}\right)+\cdots + {\rm{Log}}\left(\frac{y_{n+1,\lambda}^{\frac{1}{d^{n+1}}}}{y_{n,\lambda}^{\frac{1}{d^{n}}}}\right)+\cdots
	\end{equation}
	using (\ref{zero}), we conclude that the series (\ref{Log1}) converges and consequently the series (\ref{tele1}) also converges.

	\medskip 
	\no 
	Therefore,  
	\[
	\lim_{n\rightarrow \infty} y_{n,\lambda}^{\frac{1}{d^n}}
	\]
	exists.
	
	\medskip 
	\no 
	Now for each $\lambda\in M$, define $\phi_{H,\lambda}^+: V_R^+ \rightarrow \mathbb{C}$ as follows:
	\begin{equation}\label{Green}
	\phi_{H,\lambda}^+(x,y)={c_H}^{-\frac{1}{d-1}} \lim_{n\rightarrow \infty} y_{n,\lambda}^{\frac{1}{d^n}}.
	\end{equation} 
	Since
	\begin{eqnarray*}
		H_\lambda^{n+1}(x,y)&=&H_{\sigma_H^n(\lambda)}\left({(H_\lambda^n)}_1(x,y), {(H_\lambda^n)}_2(x,y)\right)\\
		&=&\left({\left(H_\lambda^{n+1}\right)}_1(x,y), c_H{(H_\lambda^n)}_2^d(x,y)+q_{\sigma_H^n(\lambda)}\left({(H_\lambda^n)}_1(x,y),{(H_\lambda^n)}_2(x,y)\right)\right),
	\end{eqnarray*}
	it follows that
	\begin{eqnarray*}
		{\left(H_{\sigma_H(\lambda)}^n\right)}_2(H_\lambda(x,y))&=& 
		c_H {(H_\lambda^n)}_2^d(x,y)(1+L_\lambda(x,y))
	\end{eqnarray*}
	where $L_\lambda(x,y)\rightarrow 0$ as $\lVert (x,y)\rVert\rightarrow \infty$ in $V_R^+$.
	
	\medskip 
	\no 
	Now
	\begin{eqnarray}\label{simi}
	\nonumber	\phi_{H,\sigma_H(\lambda)}^+(H_\lambda(x,y))&=&  
		c_{H}^{-\frac{1}{d-1}}\lim_{n\rightarrow \infty}{\left(H_{\sigma_H(\lambda)}^n\right)}_2^{\frac{1}{d^n}}(H_\lambda (x,y))\\
		&=&c_H \lim_{n\rightarrow \infty} c_H^{\frac{1}{d^n}}{\left(c_H^{-\frac{1}{d-1}}{(H_\lambda^n)}_2^{\frac{1}{d^n}}(x,y){\left(1+L_\lambda(x,y)\right)}^{\frac{1}{d^{n+1}}} \right)}^d .
	\end{eqnarray}
	Since $L_\lambda(x,y)\rightarrow 0$ as $\lvert y\rvert \rightarrow \infty$ in $V_R^+$ and $ c_H^{\frac{1}{d^n}}\rightarrow 1$ as $n\rightarrow \infty$, we get 
	\begin{equation*}
	\phi_{H,\sigma_H(\lambda)}^+(H_\lambda(x,y))=c_H {(\phi_{H,\lambda}^+(x,y))}^d
	\end{equation*}
	in $V_R^+$.
	
	\medskip 
	\no 
	As before, we define
	\[
	\tilde{x}_{n,\lambda}={(H_\lambda^{-n})}_1,
	\]
	a polynomial in $x$ and $y$ of degree $d^n$ with leading term ${c}_H'$. Next, define
	\begin{equation*}
	\phi_{H,\lambda}^-(x,y)={{c}_H'}^{-\frac{1}{d-1}} \lim_{n\rightarrow \infty} {(\tilde{x}_{n,\lambda})}^{\frac{1}{d^n}},
	\end{equation*}
	which makes sense for the same reason explained in the previous case.
	
	\medskip 
	\no 
	Using a same sort of arguments as in (\ref{simi}), we get that 
	\begin{equation*}
	\phi_{H,\lambda}^-\circ H_\lambda^{-1}(x,y)={c_H'} {\left(\phi_{H,\sigma_H(\lambda)}^-(x,y)\right)}^d
	\end{equation*}
	for all $(x,y)\in V_R^-$ and for all $\lambda\in M$.
\end{proof}


\subsection{Rigidity of skew products of H\'{e}non  maps with the same fibered Julia sets}
\begin{prop}\label{fibered Julia}
	Let $H$ and $F$ be skew products of H\'{e}non maps fibered over a compact metric space $M$ and let for each $\lambda\in M$, the fibered Julia sets of $H$ and $F$ are the same, i.e., $J_{H,\lambda}^\pm= J_{F,\lambda}^\pm$ for all $\lambda\in M$. Then for each $\lambda\in M$, 
	\begin{equation*}
	F_{\sigma_H(\lambda)}\circ H_\lambda = \gamma \circ H_{\sigma_F(\lambda)}\circ F_\lambda 
	\end{equation*}
	for some $\gamma:(x,y)\mapsto (\delta^+ x, \delta^- y)$ with $\lvert \delta^\pm \rvert=1$.
	Consequently,  there exists a skew  map 
	$\Gamma:(\lambda, x,y)=(\mu(\lambda), \delta^+ x, \delta^- y)$  such that $F\circ H=\Gamma \circ H \circ F$  in $M\times \mathbb{C}^2$, where $$\mu=\left(\sigma_F \circ \sigma_H \circ \sigma_F^{-1}\circ \sigma_H^{-1}\right)$$   and $\lvert \delta^\pm \rvert=1$.
\end{prop}

\begin{proof}
	Let 
	\[
	F:(\lambda,x,y)=(\sigma_F(\lambda), F_\lambda(x,y))
	\]
	and 
	\[
	H:(\lambda,x,y)=(\sigma_H(\lambda), H_\lambda(x,y)).
	\]
	
	\medskip 
	\no 
	Since for each $\lambda\in M$, the functions $G_{H,\lambda}^+$ and $G_{F,\lambda}^+$ are the pluricomplex Green functions for the sets $J_{H,\lambda}^+$ and $J_{F,\lambda}^+$ respectively (see Proposition 1.2 in \cite{PV}), we have
	\begin{equation*}
	G_{H,\lambda}^+=G_{F,\lambda}^+
	\end{equation*} 
	for all $\lambda\in M$.
	Thus combining (\ref{Green})  and (\ref{fiberedGreen}), we get that
	\begin{equation*}
	G_{H,\lambda}^+=\log \lvert \phi_{H,\lambda}^+\rvert+\frac{1}{d_H-1}\log \lvert c_H \rvert
	\end{equation*}
	and  
	\begin{equation*}
	G_{F,\lambda}^+=\log \lvert \phi_{F,\lambda}^+\rvert+\frac{1}{d_F-1}\log \lvert c_F \rvert
	\end{equation*}
	in $V_R^+$ for all $\lambda\in M$. 
	The rest of the proof follows by using  fibered B\"{o}ttcher coordinates constructed in the Proposition \ref{fiberedBott} and pursuing the same techniques as in the proof of Theorem 1.1 in \cite{rigidity}.

\end{proof}
\bibliographystyle{amsplain}

\end{document}